\documentclass{amsart} 
\usepackage{amsmath,amsthm}
\usepackage{amssymb,amscd,latexsym}
\usepackage{boxedminipage}
\usepackage{eepic}
\usepackage{epic}
\usepackage{wrapfig}
\usepackage{graphicx}
\newcommand{\erase}[1]{}
\theoremstyle{remark}
\newtheorem{remark}{Remark} 
\newtheorem{theorem}{Theorem}[section]
\newtheorem{definition}[theorem]{Definition}
\newtheorem{proposition}[theorem]{Proposition}

\newtheorem{corollary}[theorem]{Corollary}
\newtheorem{example}[theorem]{Example}
\newtheorem{lemma}[theorem]{Lemma}
\numberwithin{equation}{section}
\newcommand{\bp}{\begin{pmatrix}}
\newcommand{\ep}{\end{pmatrix}}
\newcommand{\bps}{\begin{smallmatrix}}
\newcommand{\eps}{\end{smallmatrix}}
\def\C{{\mathbb C}}

\def\R{{\mathbb R}}
\def\Z{{\mathbb Z}}

\def \0{{\bf 0}}
\def \1{{\bf 1}}

\def \mf#1#2#3#4{
\xymatrix{{#1}\  \ar@<0.4ex>[r]^{{#2}} & \ {#4}
\ar@<0.4ex>[l]^{{#3}}
}
}

\def \mfs#1#2#3#4{\!
\xymatrix@C=1.5em{{#1} \! \ar@<0.2ex>[r]^{{#2}} & \! {#4}
\ar@<0.2ex>[l]^{{#3}}
}
\!}

\def \mfl#1#2#3#4{
\xymatrix@C=2.6em{{#1}\  \ar@<0.4ex>[r]^{{#2}} &\  {#4}
\ar@<0.2ex>[l]^{{#3}}
}
}

\def \mfss#1#2#3#4{\!
\xymatrix@C=1.5em{{#1} \ar@<0.3ex>[r]^{{#2}} & {#4}
\ar@<0.3ex>[l]^{{#3}}
}
\!}

\newcommand{\deeq}{\mathbin{\hbox{$=$ \lower 1.7pt\rlap{\hskip -8.5pt .}}}} %Editor
\begin{document}
\title{Infinite examples of cancellative monoids that do not always have least common multiple.} 

\author{Tadashi Ishibe}
\begin{abstract}
We will study the presentations of fundamental groups 
of the complement of complexified real affine line arrangements that do not contain two parallel lines. By Yoshinaga's minimal presentation, we can give positive homogeneous presentations of the fundamental groups. We consider the associated monoids defined by the presentations. It turns out that, in some cases, left (resp. right) \emph{least common multiple} does not always exist. Hence, the monoids are neither \emph{Garside} nor \emph{Artin}. 
Nevertheless, we will show that they carry certain particular elements 
similar to the \emph{fundamental elements} in Artin monoids, and that, by improving the classical method in combinatorial group theory, they are \emph{cancellative monoids}. As a result, we will show that the word problem can be solved 
and the center of them are determined.
\end{abstract}

%\hmjlogo{}{}{}{}
%\address{Department of Mathematical Sciences, \\
%University of Tokyo, \\
%3-8-1 Komaba Meguro-ku Tokyo, 153-8914 Japan
%}
  
%\date{}

%\vspace{ -0.6cm} Editor

% {\renewcommand{\baselinestretch}{0.9} Editor

%  \vspace{0.3cm}

%\subjclass{20F05}

%\keywords{Artin groups; Artin monoids; Garside groups; Garside monoids}

%Editor
%\noindent
%2000 Mathematics Subject Classification Numbers:  20F05 \\
%key words and phrases:  monoid, fundamental group, the word problem, the conjugacy problem

\maketitle

\section{Introduction} 
Early in $70's$ the braid groups are generalized to a wider class of groups, the fundamental groups of the regular orbit spaces of finite reflection groups (\cite{[B]}), which are called either the \emph{Artin group} (\cite{[B-S]}) or the \emph{generalized braid group} (\cite{[De]}). In \cite{[B]}, E. Brieskorn gave a presentation of the fundamental groups by certain positive homogeneous relations, called Artin braid relations. The monoid defined by that presentation is called \emph{Artin monoid} of finite type. In \cite{[B-S]}, by refering to the method in \cite{[G]}, they showed that the Artin monoid is \emph{cancellative} (i.e. $axb = ayb$ implies $x = y$ ) and that, for any two elements in the monoid, left (resp. right) common multiples exist. Hence, due to the \"Ore's criterion, the Artin monoid of finite type injects in the corresponding Artin group. Furthermore, they showed that, for any two elements, left (resp. right) least common multiple exists (see \cite{[B-S]} \S4). By using this property, they defined a particular element $\Delta$, the \emph{fundamental element}, in the monoid. By using the injectivity and the existence of this element $\Delta$, they solved the word and conjugacy problem in the Artin groups of finite type and determined the center of them.\par 
After this work, in the late $90's$, the notion of Artin group (resp. Artin monoid) is generalized by French mathematicians (\cite{[D-P]}, \cite{[D1]}), which is called the \emph{Garside group} (resp. \emph{Garside monoid}).
The Garside group is defined as the group of fractions of a Garside monoid. A Garside monoid is a finitely generated monoid that satisfies the following conditions: $\mathrm{i})$ the monoid is cancellative; $\mathrm{ii})$ \emph{atomic} (i.e. the expressions of a given element have bounded lengths); $\mathrm{iii})$ left (resp. right) least common multiples exist; $\mathrm{iv})$ a \emph{Garside element} exists. Hence, the Garside monoid trivially satisfies the \"Ore's criterion. We note that, under the assumption that the monoid is atomic and cancellative, an element $\Delta$ in the monoid is a Garside element if and only if $\Delta$ is a \emph{fundamental element} (Proposition 2.2). For Garside group, the word problem can be solved. Moreover, the conjugacy problem can be solved (\cite{[P]}, \cite{[Ge]}), by improving the method in \cite{[G]} and \cite{[E-M]}.\par
 Since the condition $\mathrm{iii})$ is a strong assumption, some Zariski-van Kampen monoids do not satisfy the condition $\mathrm{iii})$ (\cite{[B-M]}\cite{[I1]}\cite{[I2]}\cite{[S-I]}). As far as we know, for non-abelian positive homogeneously presented monoids that do not satisfy the condition $\mathrm{iii})$, there are few examples for which the cancellativity of them has been shown, since the pre-existing technique to show the cancellativity is not perfect (\cite{[G]}\cite{[B-S]}\cite{[D2]}\cite{[D3]}). We have an important remark on the method in \cite{[D2]}, \cite{[D3]}. If presentation of a positive homogeneously presented monoid satisfies some condition, called \emph{completeness}, the cancellativity of it can be trivially checked. However, in general, the presentation of a monoid is not complete. When the presentation is not complete, in order to obtain a complete presentaion, some procedure, called \emph{completion}, is carried out. From our experience, for most of non-abelian monoids that do not satisfy the condition $\mathrm{iii})$, these procedures do not finish in finite steps. Since, for monoids of this kind, nothing is discussed in \cite{[D2]}, \cite{[D3]}, we need to improve the technique to show the cancellativity. On the other hand, the presentations of the examples $G^{+}_{\mathrm{B_{ii}}}$ (\cite{[I1]}), $G^{+}_{m, n}$ (\S3), $G^{+}_n$ and $H^{+}_n$ (\cite{[I2]}) are not complete and the procedures do not finish in finite steps. Nevertheless, in \cite{[I1]}, for the monoid, called the type $\mathrm{B_{ii}}$, that does not satisfy only the condition $\mathrm{iii})$, the author has solved the word problem and the conjugacy problem, and determined the center of it by showing the monoid injects in the corresponding group.\par
 In this article, we will construct infinite examples that do not satisfy only the condition $\mathrm{iii})$. To obtain the infinite examples, we will study the presentations of the fundamental groups of the complement of complexified real affine line arrangements that do not contain two parallel lines (\S2). By Yoshinaga's minimal presentation, we can give positive homogeneous presentations of the fundamental groups. In Section 3, we will consider a special type of line arrangement. The line arrangement consists of $m+n+1$ real affine lines. We will compute the fundamental group of the complement of it's complexification by using Zariski-van Kampen method. The same presentation can be obtained by Yoshinaga's minimal presentation. It turns out that fundamental elements exist in the associated monoid defined by the presentation (Proposition 3.1). Moreover, we will show the cancellativity of it successfully, by improving the classical method in combinatorial group theory (for instance \cite{[G]}\cite{[B-S]}) (Proposition 4.3). Due to \"Ore's criterion, the associated monoid injects in the corresponding group (Proposition 5.1). As a result, some decision problems in the group can be solved (Proposition 5.2, 5.4). We remark that the fundamental group is isomorphic to 
\[
\Z \times F_{m} \times F_{n}.
\]
Hence, from a group theoretical point of view, we may say that this fundamental group is well-known. \par
%Let us explain details of our motivation. Affine Zariski-van Kampen method is a% convenient tool for computing fundamental groups of complements to affine plan%e curves (see \cite{[Ch]}, \cite{[T-S]} for instance). It gives you the fundame%ntal groups in terms of generators and relations. In some cases, the presentati%ons of the fundamental groups are positive homogeneous presentations. Then, we %can associate monoids defined by the presentations, which are called the Zarisk%i-van Kampen monoids. When we deal with Zariski-van Kampen monoids, we often me%et the situation in which the monoids satisfy the conditions $\mathrm{ii})$ and% $\mathrm{iv})'$ but do not satisfy the condition $\mathrm{iii})$. We know that%, for the groups corresponding to the monoids satisfying the conditions $\mathr%m{i})$, $\mathrm{ii})$ and $\mathrm{iv})'$, the word problem can be solved (Lem%ma 2.3). 
%As far as we know, for non-abelian monoids that do not satisfy the condition $\%mathrm{iii})$, it is difficult to show the cancellativity of them. Thus, we nee%d to improve the technique to show the cancellativity. In Proposition 4.3, we p%artially succeed to improve the technique. In \cite{[I2]}, we have constructed %examples to which only our method can be applied.

\section{Positive Presentation}

In this section, we first recall from \cite{[B-S]} some basic definitions and notations. Secondly, for a positive  finitely presented group 
\[
G = \langle L \!\mid \!R\rangle, %Editor
\]
we associate a monoid defined by it. We will extend a basic notion in \cite{[B-S]}, \emph{fundamental element}, for a positively presented atomic monoid. Lastly, by using a fundamental element $\Delta$ in the associated monoid, we will discuss the word problem in the group $G = \langle L \!\mid \!R\rangle$.

 Let $L$ be a finite set. Let $F(L)$ be the free group generated by $L$, and let $L^*$ be the free monoid generated by $L$ inside $F(L)$. We call the elements of $F(L)$ \emph{words} and the elements of $L^*$ \emph{positive words}. The empty word $\varepsilon$ is the identity element of $L^*$.
If two words $A$, $B$ are identical letter by letter, we write $A \equiv B$.
 Let $G = \langle L \!\mid \!R\rangle$ be a positive presented group (i.e. the set $R$ of relations consists of those of the form $R_i\!=\!S_i$ 
where 
$R_i$ and $S_i$ are positive words 
 ), where $R$ is the set of relations. We often denote the images of the letters and words under the quotient homomorphism %: \\ Editor
%\centerline{ Editor
$$\ F(L)\ \longrightarrow\ G$$
%}\\ Editor
by the same symbols and the equivalence relation on elements $A$ and $B$ in $G$ is denoted by $A = B$.

 Secondly, we recall some terminologies and concepts on a monoid $M$. An element $U\!\in\! M$ is said 
to \emph{divide} $V\!\in\! M$ from the left (resp.~right), and 
denoted by $U|_lV$ (resp.~$U|_rV$), if there exists $W\!\in\! M$ 
such that $V\!=\! UW$ (resp.~$V\!=\! WU$). 
We also say that $V$ is \emph{left-divisible} (resp.\emph{right-divisible}) by $U$, or $V$ is 
a \emph{right-multiple} (resp.\emph{left-divisible}) of $U$. We say that $M$ \emph{admits the left {\rm (resp.} right{\rm )} 
divisibility theory}, if 
for any two elements $U,V$ in $M$, there always exists their 
left (resp.~right) least common multiple, i.e.~a left (resp.~right) common multiple that divides any other left (resp.~right) common multiple. 

Lastly, we consider two operations on the set of subsets of a monoid $M$. For a subset $J$ of $M$, we put
\[
\mathrm{cm}_{r}(J) := \{ u \in M \mid j \,|_l \,u ,\, \forall j \in J \},\,\,\,\,\,\,\,\,\,\,\,\,\,\,\,\,\,\,\,\,\,\,\,\,\,\,\,\,
\]
\[
\mathrm{min}_{r}(J) := \{ u \in J \mid \exists v \in J \,\,\mathrm{s.t.}\,\, v \,|_l \,u \Rightarrow v = u \},
\]
and their composition by
\[
\mathrm{mcm}_{r}(J) := \mathrm{min}_{r}(\mathrm{cm}_{r}(J)).
\]
Next, we recall from \cite{[S-I]}, \cite{[I1]} some terminologies and concepts 
on positive presented monoid. And we refer to some concepts from \cite{[D-P]}, \cite{[D1]}. \par %\\ Editor
\ \ \\
\begin{definition}
%Editor
{\it Let 
$G = \langle L \!\mid \!R\rangle$
be a positive finitely presented group, where $L$ is the set of generators 
(called alphabet) and 
$R$ is the set of relations. 
Then we associate a monoid $G^+ = {\langle L \mid R\rangle}_{mo}$ defined as the quotient 
of the free monoid $L^*$ generated by $L$ by the equivalence relation $
$ defined as follows:  \par %\\ Editor
$\mathrm{i})$ two words $U$ and $V$ in $L^*$ are called \emph{elementarily 
equivalent} if either $U \equiv V$ or $V$ is obtained from $U$ by substituting 
a substring $R_i$ of $U$ by $S_i$ where $R_i\!=\!S_i$ is a relation of $R$ 
($S_i = R_i$ is also a relation if $R_i=S_i$ is a relation), \par %\\ Editor 
$\mathrm{ii})$ two words $U$ and $V$ in $L^*$ are called \emph{equivalent}, denoted by $U \deeq V$, if there exists 
a sequence $U\! \equiv \!W_0, W_1,\ldots, W_n\! \equiv \!V $ of words in $L^*$ for $n\!\in\!\Z_{\ge0}$
such that $W_i$ is elementarily equivalent to $W_{i-1}$ for $i=1,\ldots, n$.\par

1. We say that $G^+$ is \emph{atomic}, if there exists a map:
\[
\nu\ : \ G^+\ \longrightarrow\ \Z_{\ge0}
\]
such that $\mathrm{i})$ $\nu(\alpha) = 0$\,$\Longleftrightarrow$\,$\alpha = 1$ and $\mathrm{ii})$ an inequality:
\[
\nu(\alpha\beta) \geq \nu(\alpha) + \nu(\beta)
\]
is satisfied for any $\alpha, \beta \in G^{+}$. If $G^+ = {\langle L \mid R\rangle}_{mo}$ is a positive homogeneously presented monoid (i.e. the set $R$ of relations consists of those of the form $R_i = S_i$ where $R_i$ and $S_i$ are positive words of the same length ), it is clear that $G^+$ is an atomic monoid. An element $\alpha \not= 1$ in $G^{+}$ is called an \emph{atom} if it is indecomposable, namely, $\alpha = \beta \gamma$ implies $\beta = 1$ or $\gamma = 1$.\par
2. We suppose that $G^{+}$ satisfies the condition of atomic monoid. Here, we write the set of generators $L$ by $\{ g_1, g_2, \ldots, g_{m} \}$. If, for some positive word $w(g_1, \ldots, g_{i-1}, g_{i+1}, \ldots, g_m)$ (i.e. a word that is written by the generators except $g_i$), $g_i = w(g_1, \ldots, g_{i-1}, g_{i+1}, \ldots, g_m)$ is a relation of $R$, then we call the generator $g_i$ a \emph{dummy generator}. We note that, in the set $R$, a relation that has a form of $g_i = w(g_1, \ldots, g_i, \ldots, g_m)$ must be the form $g_i = w(g_1, \ldots, g_{i-1}, g_{i+1}, \ldots, g_m)$ or a trivial form $g_i = g_i$, because we suppose here that $G^{+}$ is an atomic monoid. We denote by $L'$ the set of all dummy generators of the monoid $G^+$. We put $\widetilde{L} := L \setminus L'$. We note that, if $G^+$ is an atomic monoid, the image of the set $\widetilde{L}$ in $G^{+}$ is equal to the set of all the atoms.\par
3. We say that $G^+$ is \emph{cancellative}, if 
an equality $AXB \deeq \! AYB$  for\\
 $A, B, X, Y \!\in G^+$ implies $X \deeq \! Y$.

4. The natural homomorphism $\pi:  G^+\to G$ will be called the \emph{localization homomorphism}. 
%The image of the localization homomorphism is denoted by $\pi G^+$.

%4. An element $\Delta\in G^+$ is called \emph{quasi-central} (also see [B-S] 7.%1), if there exists a permutation $\sigma_\Delta$ of $L/\!\!\sim$(: =the image %of the set $L$ in $G^+$) such that 
%\[
%s\cdot\Delta \deeq  \Delta\cdot \sigma_\Delta(s)
%\vspace{-0.2cm} Editor
%\]
%holds for all generators $s\in L/\!\!\sim$. 
% The set of all quasi-central elements is denoted by $\mathcal{QZ}(G^+)$.
% The order of an element $\sigma_{\Delta}$ in the permutation group $\mathfrak{%S}(L/\!\sim)$ is denoted by $\mathrm{ord}(\sigma_{\Delta})$.
% Note that $\Delta^{\mathrm{ord}(\sigma_{\Delta})}$ belongs to the center $\mat%hcal{Z}(G^+)$ of the monoid $G^+$.
5. An element $\Delta$$\in G^+$ is called a \emph{Garside element}
if the sets of left- and right-divisors of $\Delta$ coincide, generate $G^{+}$, and are finite in number.

6. An element $\Delta$ in an atomic monoid $G^+$ is called a \emph{fundamental element}
if there exists a permutation $\sigma_\Delta$ of $\widetilde{L}$ such that, for any $s \in \widetilde{L}$, there exists 
$\Delta_s$$\in G^+$ satisfying the following relation: 
\[
\Delta \deeq s \cdot \Delta_s \deeq  \Delta_s \cdot \sigma_\Delta(s).
\]
We note that, if the monoid $G^{+}$ is a cancellative monoid, there exists a unique permutation $\sigma_\Delta$ for a fundamental element $\Delta$.
We denote by $\mathcal{F}(G^+)$ the set of all fundamental elements of $G^+$. The order of an element $\sigma_{\Delta}$ in the permutation group $\mathfrak{S}(\widetilde{L})$ is denoted by $\mathrm{ord}(\sigma_{\Delta})$. 
Note that $\varepsilon \not\in \mathcal{F}(G^+)$. }
%It is easy to show that 
%\[
%\mathcal{F}(G^+)\mathcal{QZ}(G^+)=\mathcal{QZ}(G^+) \mathcal{F}(G^+)= \mathcal{%F}(G^+).
%\]

%6. A fundamental element $\Delta$ is called a \emph{minimal fundamental element%} if any fundamental element dividing $\Delta$ from right or left coincides wit%h $\Delta$ itself.

%7. A quasi-central element $\Delta$ is called \emph{indecomposable}, if it does% not decompose into a product of two nontrivial quasi-central elements. We note% that the identity element $\varepsilon$ is not indecomposable. We call a funda%mental element \emph{prime}, if it is an indecomposable quasi-central element.

\end{definition}
\ \ \\
From the definitions, it follows that the notion of fundamental elements is equivalent to the notion of Garside elements. 
\ \ \\
\begin{proposition}
%\label{mythm}
{\it  Let $G = \langle L \!\mid \!R\rangle$ be a positively presented group, and let $G^+ = {\langle L \mid R\rangle}_{mo}$ be the associated monoid. Assume that the monoid $G^+$
is an atomic, cancellative monoid.\\
Then, an element $\Delta$ in $G^+$ is a fundamental element if and only if $\Delta$ is a Garside element. }
\end{proposition}
\begin{proof} Assume that $\Delta$ is a fundamental element. We put $N := \mathrm{ord}(\sigma_{\Delta})$. We decompose $\Delta$ into $U \cdot V$. We write $U$ and $V$ by $u_{1}u_{2} \cdots u_{k}$ and $v_{1}v_{2} \cdots v_{\ell}$ respectively $(u_{1}, u_{2}, \ldots, u_{k}, v_{1}, v_{2}, \ldots, v_{\ell} \in \widetilde{L})$. Since the monoid $G^{+}$ is a cancellative monoid, by the definition of $\Delta$ we have
\[
u_{1}u_{2} \cdots u_{k} \cdot v_{1}v_{2} \cdots v_{\ell} \deeq v_{1}v_{2} \cdots v_{\ell} \cdot \sigma_{\Delta}(u_{1})\sigma_{\Delta}(u_{2}) \cdots \sigma_{\Delta}(u_{k})
\]
\[
\deeq \sigma_{\Delta}(u_{1})\sigma_{\Delta}(u_{2}) \cdots \sigma_{\Delta}(u_{k}) \cdot \sigma_{\Delta}(v_{1})\sigma_{\Delta}(v_{2}) \cdots \sigma_{\Delta}(v_{\ell})\,\,\,\,\,\,\,\,\,\,\,\,\,\,\,\,\,\,\,\,\,\,\,
\]
\[
\,\,\,\,\,\,\,\,\,\,\,\,\,\,\,\,\,\deeq \sigma_{\Delta}^{N-1}(u_{1})\sigma_{\Delta}^{N-1}(u_{2}) \cdots \sigma_{\Delta}^{N-1}(u_{k}) \cdot \sigma_{\Delta}^{N-1}(v_{1})\sigma_{\Delta}^{N-1}(v_{2}) \cdots \sigma_{\Delta}^{N-1}(v_{\ell})
\]
\[
\deeq \sigma_{\Delta}^{N-1}(v_{1})\sigma_{\Delta}^{N-1}(v_{2}) \cdots \sigma_{\Delta}^{N-1}(v_{\ell}) \cdot u_{1}u_{2} \cdots u_{k}.\,\,\,\,\,\,\,\,\,\,\,\,\,\,\,\,\,\,\,\,\,\,\,\,\,\,\,\,\,\,\,\,\,\,\,\,\,\,\,
\]
Hence, the element $U$ is also a right divisor of $\Delta$.\par
Next, we assume that $\Delta$ is a Garside element. We recall that the set $\widetilde{L}$ is equal to the set of all the atoms. Here, we write $\widetilde{L}$ by $\{ s_1, s_2, \ldots, s_{m} \}$. Since $\Delta$ is a Garside element, for each $i \in \{ 1, 2, \ldots, m \}$, $s_i$ devides $\Delta$ from the left. Thus, we can associate a quotient $\Delta_{s_i}$ (i.e. $\Delta \deeq s_i \cdot \Delta_{s_i}$ holds). Since the monoid $G^{+}$ is a cancellative monoid, we remark that the element $\Delta_{s_i}$ can be determined uniquely. We show the following Claim.\\
{\bf Claim.} 
For arbitrary two atoms $s_i, s_j$$(i \not= j)$, $\Delta_{s_i}$ cannot be a substring of $\Delta_{s_j}$.
\begin{proof} We assume that there exist two words $w_1$ and $w_2$ such that $\Delta_{s_i}$ and $\Delta_{s_j}$ satisfy the following equation 
\[ 
\Delta_{s_i} \deeq w_1 \cdot \Delta_{s_j} \cdot w_2.
\] 
By substituting $\Delta_{s_i}$ by $w_1 \cdot \Delta_{s_j} \cdot w_2$, we have
\begin{equation}
 \Delta \deeq s_j \cdot \Delta_{s_j} \deeq s_i \cdot \Delta_{s_i} \deeq s_i \cdot w_1 \cdot \Delta_{s_j} \cdot w_2. 
\end{equation} 
We consider the following two cases.\par
Case 1: $w_2 = 1$\\
Due to the cancellativity, we have the following equation
\[
s_j \deeq s_i \cdot w_1.
\]
A contradiction.\par
Case 2: $w_2 \not= 1$\\
Since $\Delta$ is a Garside element, we say that, from (2.1), the element $s_i \cdot w_1 \cdot \Delta_{s_j}$ is also a right divisor. Hence, there exists a positive word $\widetilde{w_2} \not= 1$ such that 
\[
s_i \cdot w_1 \cdot \Delta_{s_j} \cdot w_2 \deeq \widetilde{w_2} \cdot s_i \cdot w_1 \cdot \Delta_{s_j}. 
\]
 Thus, due to the cancellativity, we have
\[
s_j \deeq \widetilde{w_2} \cdot s_i \cdot w_1.
\]
A contradiction. 
 \end{proof}
Since the monoid $G^{+}$ is a cancellative monoid, there exists a unique element $A$ such that
\begin{equation}
 \Delta \deeq s_i \cdot \Delta_{s_i} \deeq \Delta_{s_i} \cdot A. 
\end{equation} 
We write $A$ in the form $\alpha_{1}\alpha_{2} \cdots \alpha_{k}$ letter by letter $(\alpha_{1}, \alpha_{2}, \ldots, \alpha_{k} \in \widetilde{L})$. Assume that $k \geq 2$. Since $\Delta$ is a Garside element, we say that the element $\Delta_{s_i} \cdot \alpha_{1} \cdots \alpha_{k-1}$ is also a right divisor. Hence, there exists a positive word $B \not= 1$ such that 
\[
\Delta \deeq B \cdot \Delta_{s_i} \cdot \alpha_{1} \cdots \alpha_{k-1}. 
\]
Due to the Claim, we have a contradiction. Hence, we say that $k = 1$. From (2.2), there exists a unique permutation $\sigma_\Delta$ of $\widetilde{L}$ such that, for any $s \in \widetilde{L}$, the following relation holds: 
\[
\Delta \deeq s \cdot \Delta_s \deeq \Delta_s \cdot \sigma_\Delta(s).
\]
 \end{proof}
Lastly, we discuss the word problem in a positively presented group. 

\begin{lemma}
%\it Editor
  Let $G = \langle L \!\mid \!R\rangle$ be a positively presented group, and let 
$G^+ = {\langle L \mid R\rangle}_{mo}$ be the associated monoid. Assume that the monoid $G^+$
is an atomic, cancellative monoid and $\mathcal{F}(G^+)
\not=\emptyset$. Then:  \par %\\ Editor
 
%\noindent%Editor
 (1)  The localization homomorphism $\pi:  G^+\to G$ is injective. \par %\\ Editor
 (2)  The word problem in $G$ is solvable
 %\\ Editor
% (3)  The conjugacy problem in $G^+$ is solvable if and only if the conjugacy problem in $G$ is solvable.%
\end{lemma}

\begin{proof}(1)  Let $\Delta \in \mathcal{F}(G^+)$ be a fundamental element. We can easily show that, for any $U \in G^+$, there exists a sufficiently large integer $\ell$ such that $U$ devides $\Delta^{\ell}$ from the left and the right. Hence, we show that the monoid $G^+$ satisfies \"Ore's condition (see \cite{[C-P]}). Therefore, the localization homomorphism $\pi$ is injective.\par
(2)  We put $\Lambda : = \Delta^{\mathrm{ord}(\sigma_{\Delta})}$, which belongs to the center $\mathcal{Z}(G^+)$ of the monoid $G^+$. 
For any two elements $U, V$ in $G$, there exists a non-negative integer $k$ in $\Z_{\ge0}$ such that both ${(\pi(\Lambda))}^{k} U$ and ${(\pi(\Lambda))}^{k} V$ are equivalent to positive words. Since the localization homomorphism $\pi$ is injective, there exists a unique element $U' \in G^+$ (resp. $V' \in G^+$) such that
\[
\pi(U')={(\pi(\Lambda))}^{k} U (\rm{resp}. \pi(V')={(\pi(\Lambda))}^{k} V).
\] 
 Therefore, we can show that $U = V$ can be shown in $G$ algorithmically if and only if $U' \deeq  V'$ can be shown in $G^+$ algorithmically. Because the monoid $G^+$ is an atomic monoid, we can obtain algorithmically all the possible expressions of two words  $U'$ and $V'$ in $G^+$ in a finite number of steps. Hence, by comparing two types of complete lists of all the possible expressions of words $U'$ and $V'$, we decide in a finite number of steps whether  $U' \deeq V'$ or not. Consequently, the word problem in $G$ can be solved.
%(3)  If two elements $U$ and $V$ in $G$ are conjugate, then there exists a word% $B$ such that $BU = VB$. There exists a non-negative integer $l$ in $\Z_{\ge0}%$ such that ${(\pi(\Lambda))}^{l} B$ is equivalent to a positive word. Since ${%\pi(\Lambda)}$ belongs to the center of the group $G$, we say that two elements% $U$ and $V$ in  $G$ are conjugate precisely when there is a positive word $A$ %such that $AU$ is equivalent to $VA$. Therefore, due to the injectivity of the %localization homomorphism $\pi$, we can show that the conjugacy problem in $G^+%$ is solvable if and only if the conjugacy problem in $G$ is solvable.
\end{proof}
Here is an important observation on the existence of fundamental elements in the monoid associated with the presentation of fundamental group of the complement of line arranement that is given by Yoshinaga's minimal presentation (\cite{[Y]}). Let $\mathcal{A} = \{ \ell_1, \ell_2, \ldots, \ell_{N} \}$ be a real line arrangement in $\R^2$ that does not contain two parallel lines and is equipped with an oriented generic flag $\mathcal{F}^{0} \subset \mathcal{F}^{1} \subset \mathcal{F}^{2} = \R^2$. By Yoshinaga's minimal presentation, we give a positive homogeneous presentation of the fundamental group $\pi_1(M(\mathcal{A}))$
. Here, we write the generator system by $\{ \gamma_1, \gamma_2, \ldots, \gamma_{N} \}$. When we take the generic line $\mathcal{F}^{1}$ far away from all the intersection points, we can show that, in the finitely presented group, a cyclic defining relation $\left[ \gamma_{1}, \gamma_{2}, \ldots, \gamma_{N} \right]$ holds. As a corollary, we have the following statement.
\ \ \\
\begin{corollary}
%\label{mythm}
%\it Editor
{\it  An element $\Delta := \gamma_1 \gamma_2 \cdots \gamma_N$ in the associated monoid is a fundamental element. }
\end{corollary}
Due to the Lemma 2.3, if the cancellativity of the associated monoid is proved, we can solve the word problem in the presented group. Hence, to show the cancellativity of the associated monoids is important for an understanding of the corresponding fundamental groups. If the associated monoid is not a cancellative monoid (i.e. a relation $\alpha \beta \deeq \!\alpha \gamma$ holds but $\beta \deeq \!\gamma$ does not hold ), we add the relation $\beta \deeq\! \gamma$ to the list of original defining relations. Then, we expect that the new monoid is a cancellative monoid. Even if the new monoid is not a cancellative monoid, by adding more new relations to the list each time, we expect that, in a finite number of steps, we can find a cancellative monoid. Contrary to our expectation, there are  interesting examples, where the above process cannot finish in a finite number of steps.
\ \ \\
\begin{figure}
\begin{picture}
(190,220)(110,-30)
\put(112,85){\line(3,-1){197}}
\put(90,33){\line(1,0){219}}
\put(201,13){\line(-1,4){38}}
\put(94,25){\line(4,1){179}}
\put(187,159){\line(-1,-2){76}}
\put(140,161){\line(1,-1){150}}
\put(180,4){\line(0,1){10}}
\put(180,19){\line(0,1){10}}
\put(180,34){\line(0,1){10}}
\put(180,49){\line(0,1){10}}
\put(180,64){\line(0,1){10}}
\put(180,79){\line(0,1){10}}
\put(180,94){\line(0,1){10}}
\put(180,109){\line(0,1){10}}
\put(180,124){\line(0,1){10}}
\put(180,139){\line(0,1){10}}
\put(180,154){\line(0,1){10}}
\put(180,169){\line(0,1){10}}
\put(314,31){$\ell_{a}$}
\put(278,68){$\ell_{b}$}
\put(310,16){$\ell_{c}$}
\put(292,8){$\ell_{e}$}
\put(196,6){$\ell_{d}$}
\put(191,157){$\ell_{f}$}
\put(173,-6){$\mathcal{F}^{1}$}
\put(172,129){$\circle*{3}$}
\put(268,33){$\circle*{3}$}
\put(124,33){$\circle*{3}$}
\end{picture}
\caption{a line arrangement $\mathcal{A}_{6}$}
\label{fig1}
\end{figure}
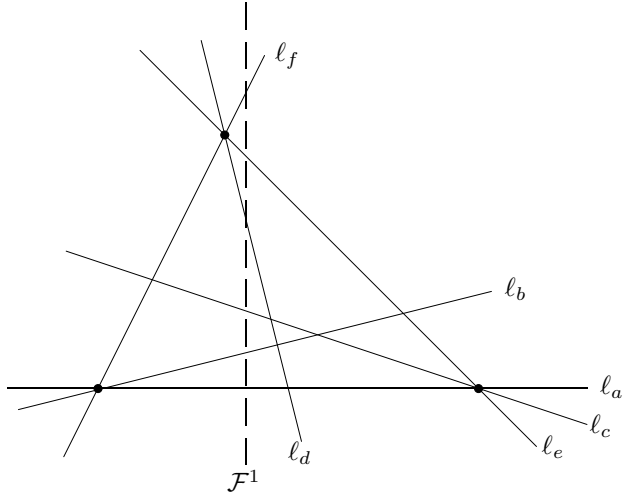
\begin{example}
%\label{mythm}
%\it Editor
Let $\mathcal{A}_{6}= \{ \ell_{a}, \ell_{b}, \ldots, \ell_{f} \}$ be the line arrangement that is written in Figure 1 and let $\mathcal{F}^{1}$ be a generic line. By Yoshinaga's minimal presentation, we give the following positive homogeneous presentation: 
\[\begin{array}{rlll}
  \pi_1(M(\mathcal{A}_{6})) 
&\cong&
\Biggl{\langle}\!
a, b, c, d, e, f\,
\biggl{|}
\begin{array}{lll}
  abf = bfa = fab, ace = cea = eac ,\\
 def = efd = fde, ad=da, cd=dc,\\
  bc=cb, bd=db, be=eb, cf=fc
 \end{array}\
\Biggl{\rangle}\ .
\end{array}\ 
\]
For the above positive homogeneous presented group, we associate the monoid 
$M_6$. We show the following Claim.\\
{\bf Claim.}  
In the monoid $M_6$, $cdea^{k}f \deeq cea^{k}fd$ holds but $dea^{k}f \deeq ea^{k}fd$ does not hold ($ k = 1, 2, \ldots $). Moreover, $bfe^{k}ac \deeq fe^{k}abc$ holds but $bfe^{k}a \deeq fe^{k}ab$ does not hold, and $cef^{k}ab \deeq ef^{k}acb$ holds but $cef^{k}a \deeq ef^{k}ac$ does not hold ($ k = 1, 2, \ldots $).
\begin{proof} In the monoid $M_6$, we have
\[
cdea^{k}f \deeq dcea^{k}f \deeq da^{k}cef \deeq a^{k}cdef \deeq a^{k}cefd \deeq cea^{k}fd. 
\]
However, we cannot show the relation $dea^{k}f \deeq ea^{k}fd$ by using only the above defining relations.
\end{proof}

\end{example}
%In general, a minimal fundamental element may not be prime. Here is an example.\begin{figure}
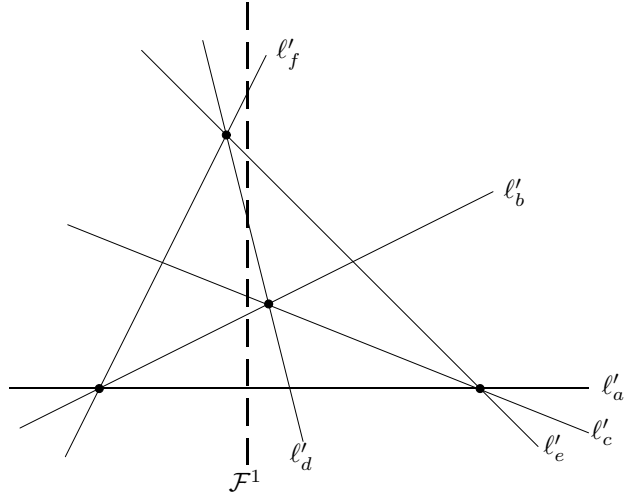
\begin{figure}
\begin{picture}(190,220)(120,-30)
%\put(121,130){\line(3,-2){179}}
\put(112,95){\line(5,-2){197}}
\put(90,33){\line(1,0){219}}
\put(201,13){\line(-1,4){38}}
\put(94,18){\line(4,2){179}}
%\put(102,11){\line(4,4){106}}
\put(187,159){\line(-1,-2){76}}
\put(140,161){\line(1,-1){150}}
\put(180,4){\line(0,1){10}}
\put(180,19){\line(0,1){10}}
\put(180,34){\line(0,1){10}}
\put(180,49){\line(0,1){10}}
\put(180,64){\line(0,1){10}}
\put(180,79){\line(0,1){10}}
\put(180,94){\line(0,1){10}}
\put(180,109){\line(0,1){10}}
\put(180,124){\line(0,1){10}}
\put(180,139){\line(0,1){10}}
\put(180,154){\line(0,1){10}}
\put(180,169){\line(0,1){10}}
%\put(112,169){\rotatebox{30}{gnu}}
%\put(113,5)\scalebox{2}{A}
\put(314,31){$\ell'_{a}$}
\put(277,105){$\ell'_{b}$}
\put(310,14){$\ell'_{c}$}
\put(292,8){$\ell'_{e}$}
\put(196,5){$\ell'_{d}$}
\put(191,158){$\ell'_{f}$}
\put(173,-6){$\mathcal{F}^{1}$}
\put(172,129){$\circle*{3}$}
\put(268,33){$\circle*{3}$}
\put(124,33){$\circle*{3}$}
\put(188,65){$\circle*{3}$}
\end{picture}
\caption{a line arrangement $\mathcal{A'}_{6}$}
\label{fig2}
\end{figure}

\begin{example}
%\label{mythm}
%\it Editor
Let $\mathcal{A'}_{6}= \{ \ell'_{a}, \ell'_{b}, \ldots, \ell'_{f} \}$ be the line arrangement that is written in Figure 2 and let $\mathcal{F}^{1}$ be a generic line. By Yoshinaga's minimal presentation, we give the following positive homogeneous presentation: 
\[\begin{array}{rlll}
  \pi_1(M(\mathcal{A'}_{6})) 
&\cong&
\Biggl{\langle}\!
a, b, c, d, e, f\,
\biggl{|}
\begin{array}{lll}
  abf = bfa = fab, bcd = cdb = dbc ,\\
 def = efd = fde, ad=da, cf=fc, \\
   be=eb, abce=eabc, cdea=acde\\
  \end{array}\
\Biggl{\rangle}\ .
\end{array}\ 
\]
For the above positive homogeneous presented group, we associate the monoid $M'_6$. We immediately show the following Claim.\\
{\bf Claim.}  
In the monoid $M'_6$, $dbcefa \deeq dbefac$ holds but $cefa \deeq efac$ does not hold.
\begin{proof} In the monoid $M'_6$, we have
\[
dbcefa \deeq bcdefa \deeq bcfdea \deeq bfcdea \deeq bfacde \deeq fabcde \,\,\,\,\,\,\,
\]
\[
 \deeq fadbce \deeq fdabce \deeq fdeabc \deeq defabc \deeq debfac \deeq dbefac.
\]
However, we cannot show the relation $cefa \deeq efac$ by using only the defining relations of the monoid $M'_6$.
\end{proof}
We add the relation $cefa \deeq efac$ to the list of original defining relations. We consider the associated monoid $\widetilde{M'_{6}}$. Then, we can find the following infinite new relations.\\
{\bf Claim.}  
In the monoid $\widetilde{M'_{6}}$, $acde^{k+1}abf \deeq de^{k}aabcef$ holds but $acde^{k+1}ab \deeq de^{k}aabce$ does not hold ($ k = 1, 2, \ldots $). Moreover, $cefa^{k+1}cdb \deeq fa^{k}ccdeab$ holds but $cefa^{k+1}cd \deeq fa^{k}ccdea$ does not hold , and $eabc^{k+1}efd \deeq bc^{k}eefacd$ holds but $eabc^{k+1}ef \deeq bc^{k}eefac$ does not hold ($ k = 1, 2, \ldots $).
\begin{proof} In the monoid $\widetilde{M'_{6}}$, we have
\[
de^{k}aabcef \deeq de^{k}aeabcf \deeq de^{k}aeabfc \deeq de^{k}aebfac \deeq de^{k}abefac \deeq de^{k}abcefa
\]
\[
\deeq  dabce^{k}efa \deeq adbce^{k}efa \deeq acdbe^{k}efa \deeq acde^{k}ebfa \deeq acde^{k}eabf.\,\,\,\,\,\,\,\,\,\,\,\,\,\,\,\,\,\,\,\,\,\,\,\,\,\,\,\,\,\,
\]
However, we cannot show the relation $acde^{k+1}ab \deeq de^{k}aabce$ by using only the defining relations of the monoid $\widetilde{M'_{6}}$.
\end{proof}

\end{example}

\section{A Zariski-van Kampen presentation}

\begin{figure}
\begin{picture}(150,170)(120,0)
%\put(121,130){\line(3,-2){179}}
\put(112,85){\line(3,-1){197}}
\put(90,33){\line(1,0){219}}
\put(94,25){\line(4,1){179}}
%\put(102,11){\line(4,4){106}}
\put(187,159){\line(-1,-2){76}}
\put(125,176){\line(1,-1){170}}
\put(133,4){\line(0,1){10}}
\put(133,19){\line(0,1){10}}
\put(133,34){\line(0,1){10}}
\put(133,49){\line(0,1){10}}
\put(133,64){\line(0,1){10}}
\put(133,79){\line(0,1){10}}
\put(133,94){\line(0,1){10}}
\put(133,109){\line(0,1){10}}
\put(133,124){\line(0,1){10}}
\put(133,139){\line(0,1){10}}
\put(133,154){\line(0,1){10}}
\put(133,169){\line(0,1){10}}
\put(79,30){$\ell_0$}
\put(278,67){$\ell_{1}^{+}$}
\put(192,162){$\ell_{m}^{+}$}
\put(98,87){$\ell_{1}^{-}$}
\put(109,176){$\ell_{n}^{-}$}
\put(8,33){\vector(1,0){40}}
\put(29,16){\vector(0,1){40}}
%\put(112,169){\rotatebox{30}{gnu}}
%\put(113,5)\scalebox{2}{A}
%\put(270,5){B}
\put(101,16){$\circle*{1}$}
\put(98,19){$\circle*{1}$}
\put(104,13){$\circle*{1}$}
\put(298,11){$\circle*{1}$}
\put(301,14){$\circle*{1}$}
\put(304,17){$\circle*{1}$}
\put(61,7){$m$-lines}
\put(107,-4){the reference fiber}
\put(309,0){$n$-lines}
\put(90,37){P}
\put(270,37){Q}
\put(50,32){$x$}
\put(27,60){$y$}
\put(96,37){$(-a, 0)$}
\put(277,37){$(a, 0)$}
\put(268,33){$\circle*{3}$}
\put(124,33){$\circle*{3}$}
\put(297,8){\rotatebox{-45}{\huge{\}}}}
\put(87,5){\rotatebox{45}{\huge{\{}}}
\end{picture}
\caption{a line arrangement $\mathcal{A}_{m, n}$}
\label{fig3}
\end{figure}
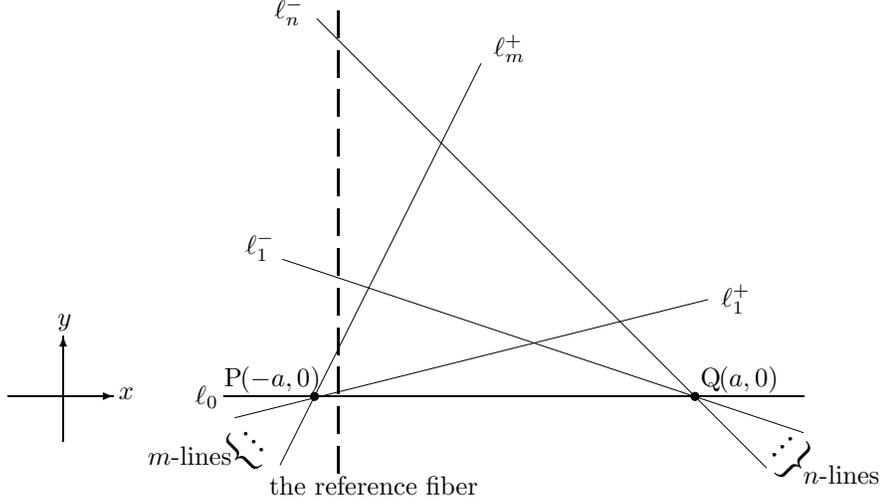
In this section, we give a Zariski-van Kampen presentation of the fundamental groups of the complement of a certain complexified real affine line arrangement. We easily show that the same presentation can be obtained by Yoshinaga's minimal presentation. Next, for the presented group, we associate a monoid defined by it. And we 
show the existence of a fundamental element in it. \par
%Lastly, we discuss the word problem in the group and determine the center of it%.
Let $\mathcal{A}_{m, n} = \{ \ell_0, \ell_{1}^{+}, \ldots, \ell_{m}^{+}, \ell_{1}^{-}, \ldots, \ell_{n}^{-} \}$ be a real line arrangement in $\R^2$ with coordinates $(x, y)$ (see Figure 2).
The line $\ell_0$ denotes the horizontal line in Figure 2. And we fix two points $P (-a, 0)$ and $Q(a, 0)$ $(a>0)$ on the line $\ell_0$. For $i \in \{ 1, \ldots, m \}$, the line $\ell_{i}^{+}$ denotes the line that passes through the point $P$ and has a positive slope $k_{i}^{+}$ $(0 < k_{1}^{+} < \cdots < k_{m}^{+})$. And, for $i \in \{ 1, \ldots, n \}$, the line $\ell_{i}^{-}$ denotes the line that passes through the point $Q$ and has a negative slope $k_{i}^{-}$ $(0 > k_{1}^{-} > \cdots > k_{n}^{-})$. All the multiple points (i.e. points where more than two lines are intersected) are two points $P$ and $Q$. We consider its complexification $\mathcal{A}_{m, n}^{\C} = \{ \ell_{0} \otimes \C, \ell_{1}^{+} \otimes \C, \ldots, \ell_{m}^{+} \otimes \C, \ell_{1}^{-} \otimes \C, \ldots, \ell_{n}^{-} \otimes \C \}$. We set
\[
M(\mathcal{A}_{m, n}) = \C^2 - ((\ell_{0} \otimes \C) \cup (\bigcup_{i
=1}^{m} \ell_{i}^{+} \otimes \C ) \cup (\bigcup_{i
=1}^{n} \ell_{i}^{-} \otimes \C )).
\]
\par
\begin{figure}
\begin{picture}(150,170)(120,0)
%\put(121,130){\line(3,-2){179}}
%\put(112,85){\line(3,-1){197}}
\put(90,123){\vector(1,0){219}}
\put(207,0){\vector(0,1){180}}
%\put(94,25){\line(4,1){179}}
%\put(102,11){\line(4,4){106}}
\put(207,22){\line(3,5){59}}
\put(207,22){\line(2,5){39}}
\put(206,23){\line(-4,5){78}}
\put(207,24){\line(-1,5){19}}
\put(207,24){\line(-2,5){38}}
%\put(298,11){$\circle*{1}$}
%\put(301,14){$\circle*{1}$}
\put(204,15){\Large{*}}
%\put(66,7){$m$-lines}
\put(313,122){\large{Re}}
\put(203,184){\large{Im}}
\put(229,127){$\circle*{1}$}
\put(225,127){$\circle*{1}$}
\put(221,127){$\circle*{1}$}
\put(216,127){$\circle*{1}$}
\put(212,127){$\circle*{1}$}
\put(208,127){$\circle*{1}$}
\put(154,127){$\circle*{1}$}
\put(150,127){$\circle*{1}$}
\put(146,127){$\circle*{1}$}
\put(141,127){$\circle*{1}$}
\put(137,127){$\circle*{1}$}
\put(133,127){$\circle*{1}$}
\put(266,130){$s$}
\put(240,131){$t_1$}
\put(180,131){$t_m$}
\put(163,131){$u_1$}
\put(120,131){$u_n$}
%\put(176,105){P}
\put(268,123){$\circle*{2}$}
\put(268,123){$\circle{8}$}
\put(267,127){$\vector(-1,0){1}$}
\put(245,123){$\circle*{2}$}
\put(245,123){$\circle{8}$}
\put(244,127){$\vector(-1,0){1}$}
\put(188,123){$\circle*{2}$}
\put(188,123){$\circle{8}$}
\put(187,127){$\vector(-1,0){1}$}
\put(167,123){$\circle*{2}$}
\put(167,123){$\circle{8}$}
\put(166,127){$\vector(-1,0){1}$}
\put(124,123){$\circle*{2}$}
\put(124,123){$\circle{8}$}
\put(123,127){$\vector(-1,0){1}$}

%\put(297,8){\rotatebox{-45}{\huge{\}}}}
%\put(87,5){\rotatebox{45}{\huge{\{}}}
\end{picture}
\caption{a generator system}
\label{fig4}
\end{figure}
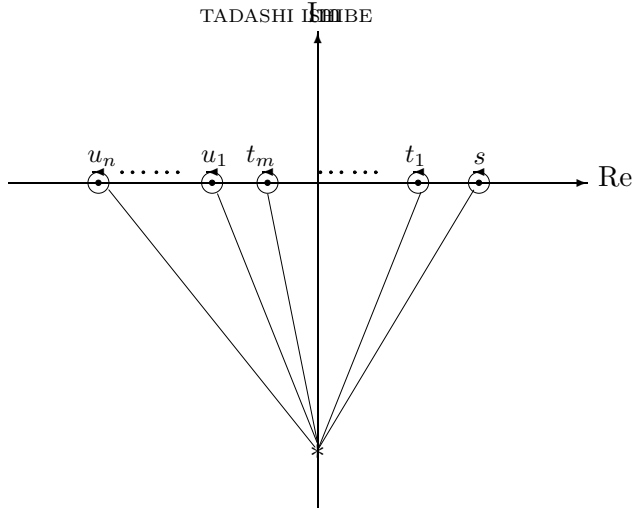
By using the Zariski-van Kampen method (see \cite{[Ch]}, \cite{[T-S]} for instance), we give a presentation of the fundamental group of the complement of the line arrangement $\mathcal{A}_{m, n}^{\C}$. We specify the technical data that are used in the computation. The dotted line in the Figure 3 denotes the reference fiber. And we have taken a generator system naturally in the reference fiber (see Figure 4). The presentation is the following:\par
\[\begin{array}{rlll}
 \pi_1(M(\mathcal{A}_{m, n})) 
&\cong&
\Biggl{\langle}\!
s,t_1, \ldots, t_m, u_1, \ldots, u_n
\biggl{|}
\begin{array}{lll}
 \left[ s, t_1, \ldots, t_m \right], \left[ s, u_1, \ldots, u_n \right], \\
\left[ t_i, u_j \right] (i=1, \ldots, m, j=1, \ldots, n) \\
  \end{array}\
\Biggl{\rangle}\ ,
\end{array}\
\]
$\mathrm{where a symbol}$ $\left[ x_{i_1}, x_{i_2}, \ldots, x_{i_k} \right]$ denotes the cyclic relations:
\[
x_{i_1} x_{i_2} \cdots x_{i_k} = x_{i_2} \cdots x_{i_k}x_{i_1} = x_{i_k}x_{i_1} \cdots x_{i_{k-1}}.
\]
We have a remark on the group $\pi_1(M(\mathcal{A}_{m, n}))$.
\ \ \\
\begin{remark}{\it  Let $\{ \overline{\ell_0}, \overline{\ell_{1}^{+}}, \ldots, \overline{\ell_{m}^{+}}, \overline{\ell_{1}^{-}} , \ldots, \overline{\ell_{n}^{-}} \}$ be projectivization of the line arrangement $\{ \ell_0, \ell_{1}^{+}, \ldots, \ell_{m}^{+}, \ell_{1}^{-}, \ldots, \ell_{n}^{-} \}$. We add the line at infinity  $\ell_{\infty}$ to the list. After carrying out a projective transformation of $\{ \overline{\ell_0}, \overline{\ell_{1}^{+}}, \ldots, \overline{\ell_{m}^{+}}, \overline{\ell_{1}^{-}} , \ldots, \overline{\ell_{n}^{-}}, \ell_{\infty} \}$ that transforms the line $\overline{\ell_0}$ to the position of the line at infinity, we consider an affinization of the arrangement. We write it by 
\[
\widetilde{\mathcal{A}_{m, n}} = \{ \widetilde{\ell_{1}^{+}}, \ldots, \widetilde{\ell_{m}^{+}}, \widetilde{\ell_{1}^{-}} , \ldots, \widetilde{\ell_{n}^{-}}, \widetilde{\ell_{\infty}} \}.
\]
 We say that
\[
\pi_1(M(\mathcal{A}_{m, n}))
\cong
\pi_1(M(\widetilde{\mathcal{A}_{m, n}})).
\]
By the theorem of Oka and Sakamoto (\cite{[O-S]}), we show the following
\[
\pi_1(M(\widetilde{\mathcal{A}_{m, n}}))
\cong
\Z \times F_{m} \times F_{n}.
\]
From a group theoretical point of view, this group is understood well. }
\end{remark}
\ \ \\
In this paper, we denote this presented group by $G_{m, n}$. For the presented group $G_{m, n}$, we associate the monoid $G_{m, n}^{+}$.\\
Next, we show the existence of a fundamental element in the monoid $G_{m, n}^{+}$.\par 
\begin{proposition}
%\label{mythm}
%\it Editor
{\it An element $\Delta := s \cdot t_1 \cdots t_m \cdot u_1 \cdots u_n$ in the monoid $G_{m, n}^{+}$ is a fundamental element.}
 \par 
\end{proposition}
\begin{proof} By using the defining relations repeatedly, we show the cyclic relations $\left[ s, t_1, \ldots, t_m, u_1, \ldots, u_n \right]$. Hence, we show that $\Delta = s \cdot t_1 \cdots t_m \cdot u_1 \cdots u_n \in \mathcal{F}(G_{m, n}^+)$.

\end{proof}

\section{Cancellativity of the monoid $G_{m, n}^{+}$}

 In this section, we prove the cancellativity of the monoid $G_{m, n}^{+}$, by improving the classical method in combinatorial group theory (for instance \cite{[G]}\cite{[B-S]}).\par
Before continuing further, we prepare notation. We put 
\[
\Delta_1 := s \cdot t_1 \cdots t_m,\,\, \Delta_2 := s \cdot u_1 \cdots u_n,
\]
\[
I_1 := \{1, \ldots, m \},\,\, I_2 :=\{1, \ldots, n \},
\]
\[
L_{0} := \{\,s, t_1, \ldots, t_m, u_1, \ldots, u_n\, \}, L_1 := \{\, t_1, \ldots, t_m \, \},\,\, L_2 :=\{\, u_1, \ldots, u_n\, \},
\]
\[
F^{+}_1 := F^{+}(\underline{t}),\,\, F^{+}_2 := F^{+}(\underline{u}),
\]
\[
F^{+}_{1, \mathrm{rm}} := \{ w(\underline{t}) \in  F^{+}_1 \mid  (t_1 \cdots t_m) \not|_r w(\underline{t}) \},
\]
\[
F^{+}_{2, \mathrm{rm}} := \{ w(\underline{u}) \in  F^{+}_2 \mid  (u_1 \cdots u_n) \not|_r w(\underline{u}) \},
\]
\[
F^{+}_{1, \mathrm{cons}} := \{ w \in  F^{+}_1 \mid  \exists i_0, j_0 \in I_1 (i_0 \leq j_0)\,\, \mathrm{s.t.}\, w = t_{i_0}t_{i_0 + 1} \cdots t_{j_0}\},
\]
\[
F^{+}_{2, \mathrm{cons}} := \{ w \in  F^{+}_2 \mid  \exists i_0, j_0 \in I_2 (i_0 \leq j_0)\,\, \mathrm{s.t.}\, w = u_{i_0}u_{i_0 + 1} \cdots u_{j_0}\}.
\]
 For arbitrary element $w(\underline{t})$ in $F^{+}_{1}$ and $w(\underline{u})$ in $F^{+}_{2}$, we put
\[
\mathrm{Div}_1(w(\underline{t})) := \{ w \in  F^{+}_{1, \mathrm{cons}}  \mid  w \,|_r \,w(\underline{t}) \},
\]
\[
\mathrm{Div}_2(w(\underline{u})) := \{ w \in  F^{+}_{2
, \mathrm{cons}}  \mid  w \,|_r \,w(\underline{u}) \}.
\]
We remark that there exists a unique element $w_{0, 1}$ in $\mathrm{Div}_1(w(\underline{t}))$ (resp. $w_{0, 2}$ in $\mathrm{Div}_2(w(\underline{u}))$ ) such that $w_1 \,|_r \,w_{0, 1}$ for any element $w_1$ in $\mathrm{Div}_1(w(\underline{t}))$ (resp. $w_2 \,|_r \,w_{0, 2}$ for any element $w_2$ in $\mathrm{Div}_2(w(\underline{u}))$ ). We put
\[
\mathrm{C}(w(\underline{t})):= w_{0,1}, \mathrm{C}(w(\underline{u})):= w_{0,2}.
\]
In view of the defining relations of $G^+_{m, n}$, there exists an element  $w'(\underline{t})$ in $F^{+}_1$ (resp. $w'(\underline{u})$ in $F^{+}_2$) such that we have a decomposition $w(\underline{t}) \equiv w'(\underline{t}) \mathrm{C}(w(\underline{t}))$ (resp. $w(\underline{u}) \equiv w'(\underline{u}) \mathrm{C}(w(\underline{u})))$ in $G^+_{m, n}$. We put
\[
\mathrm{R}(w(\underline{t})):= w'(\underline{t}), \mathrm{R}(w(\underline{u})):= w'(\underline{u}).
\]
 For arbitrary element $w(\underline{t})$ in $F^{+}_{1, \mathrm{cons}}$ (resp. $w(\underline{u})$ in $F^{+}_{2, \mathrm{cons}}$), we say that, in the monoid $G^+_{m, n}$, $w(\underline{t}) \,|_r \,\Delta_1$  (resp. $w(\underline{u}) \,|_r \,\Delta_2$). Since the quotient can be uniquely determined respectively, we denote it by $\Delta_{1, w(\underline{t})}$ (resp. $\Delta_{2, w(\underline{u})}$).
\ \ \\
\begin{theorem}
%\label{mythm}
%\it Editor
{\it The monoid $G_{m, n}^{+}$ is a cancellative monoid. }
 \par 
\end{theorem}
\begin{proof} First, we remark on the following.\par
\begin{proposition}
%\label{mythm}
%\it Editor
{\it The left cancellativity on $G_{m, n}^{+}$ implies the right cancellativity.} 
 \par 
\end{proposition}
\begin{proof} Consider a map
 $\varphi:G^{+}_{m, n}\rightarrow G^{+}_{m, n}$,
 $W\mapsto \varphi(W):=\sigma$$(rev(W))$, where $\sigma$ is a
 permutation $\big(^{\,s\,\, t_1\,\,\, \cdots \,\, t_m\,\,u_1\,\,\,\cdots \,\,u_n}_{\,s\,\, t_m\,\cdots \,\,\, t_1\,\,\,u_n\,\,\cdots\,\,\, u_1}\big)$ and $rev(W)$
 is the reverse of the word $W=x_1 x_2 \cdots x_k$ ($x_i$ is a letter) given by the word  $x_k  \cdots x_2 x_1$. In
 view of the defining relation of $G^+_{m, n}$, $\varphi$ is well-defined and is an anti-isomorphism. If $\beta \alpha \deeq \!\! \gamma \alpha$, then $\varphi(\beta \alpha) \deeq \! \varphi(\gamma \alpha)$, i.e.,  $\varphi(\alpha) \varphi(\beta) \deeq \varphi(\alpha)\varphi(\gamma)$. Using the left cancellativity, we obtain $\varphi(\beta) \deeq \! \varphi(\gamma)$ and, hence, $\beta \deeq \! \gamma $.

\end{proof}
The following is sufficient to show the left cancellativity on $G^+_{m, n}$.

\begin{proposition}
%\label{mythm}
%\it Editor
{\it Let $X$ and $Y$ be positive words in $G^+_{m, n}$ of length $r\in \Z_{\ge0}$ and let $Y^{(h)}$ be a positive word in $G^+_{m, n}$ of length $h \in \{\,0, \ldots, r \}$.
\smallskip
\\
{\rm (i)}\, If $vX \deeq \! vY$ for some $v \in L_0$, then $X \deeq \! Y$.\\
%{\rm  (ii)}\ If $s_{i}X \deeq s_{j}Y$ $(\,i \not= j)$, then there exists $w(\un%derline{u})$ in $F^{+}_{2, \mathrm{rm}}$ such that $X \deeq w(\underline{u}) \c%dot \Delta_{1, s_i} \cdot Z$, $Y \deeq w(\underline{u}) \cdot \Delta_{1, s_j} \%cdot Z$ for some positive word $Z$.\\
%{\rm (iii)}\ If $u_{i}X \deeq u_{j}Y$ $(\,i \not= j)$, then there exists $w(\un%derline{s})$ in $F^{+}_{1, \mathrm{rm}}$ such that $X \deeq w(\underline{s}) \c%dot \Delta_{2, u_i} \cdot Z$, $Y \deeq w(\underline{s}) \cdot \Delta_{2, u_j} \%cdot Z$ for some positive word $Z$.\\
%{\rm (iv)}\ If $tX \deeq s_{i}Y$, then $X \deeq \Delta_{1, t} Z$, $Y \deeq \Del%ta_{1, s_i} Z$ for some positive word $Z$.\\
%{\rm (v)}\ If $tX \deeq u_{i}Y$, then $X \deeq \Delta_{2, t} Z$, $Y \deeq \Delt%a_{2, u_i} Z$ for some positive word $Z$.\\
{\rm (ii)}\, If $t_{i} X \deeq u_{j}Y$ $(t_i \in L_1, u_j \in L_2)$, then $X \deeq u_{j} Z$, $Y \deeq t_{i} Z$ for some positive word $Z$.\\
{\rm (iii)}\, If $sX \deeq w(\underline{t}) Y^{(h)}$ for some positive word $w(\underline{t})$ of length $r-h+1$ in $F^{+}_{1}$, then $X \deeq \Delta_{1, s} \cdot \mathrm{R}(w(\underline{t})) \cdot Z$, $Y^{(h)} \deeq \Delta_{1, \mathrm{C}(w(\underline{t}))} \cdot Z$ for some positive word $Z$.\\
{\rm (iv)}\, If $sX \deeq w(\underline{u}) Y^{(h)}$ for some positive word $w(\underline{u})$ of length $r-h+1$ in $F^{+}_{2}$, then $X \deeq \Delta_{2, s} \cdot \mathrm{R}(w(\underline{u})) \cdot Z$, $Y^{(h)} \deeq \Delta_{2, \mathrm{C}(w(\underline{u}))} \cdot Z$ for some positive word $Z$.\\
{\rm (v)}\, If $t_{i}X \deeq w(\underline{t}) Y^{(h)}$ for some $t_i$ in $L_1$ and some positive word $w(\underline{t})$ of length $r-h+1$ in $F^{+}_{1}$ that satisfies $t_i \,\not|_l \,w(\underline{t})$, then there exists $w(\underline{u})$ in $F^{+}_{2, \mathrm{rm}}$ such that $X \deeq w(\underline{u}) \cdot \Delta_{1, t_{i}} \cdot \mathrm{R}(w(\underline{t})) \cdot Z$, $Y^{(h)} \deeq w(\underline{u}) \cdot \Delta_{1, \mathrm{C}(w(\underline{t}))} \cdot Z$ for some positive word $Z$.\\
{\rm (vi)}\, If $u_{i}X \deeq w(\underline{u}) Y^{(h)}$ for some $u_i$ in $L_2$ and some positive word $w(\underline{u})$ of length $r-h+1$ in $F^{+}_{2}$ that satisfies $u_i \,\not|_l \,w(\underline{u})$, then there exists $w(\underline{t})$ in $F^{+}_{1, \mathrm{rm}}$ such that $X \deeq w(\underline{t}) \cdot \Delta_{2, u_{i}} \cdot \mathrm{R}(w(\underline{u})) \cdot Z$, $Y^{(h)} \deeq w(\underline{t}) \cdot \Delta_{2, \mathrm{C}(w(\underline{u}))} \cdot Z$ for some positive word $Z$.}\\
 \par 
\end{proposition}
\begin{proof} We will show the general theorem, by refering to the double induction (see \cite{[G]}, \cite{[B-S]}, \cite{[S-I]} for instance). The theorem for positive words $X$, $Y$ of word-length $r$ and $Y^{(h)}$ of word-length $h \in \{\,0, \ldots, r \}$ will be refered to as $\mathrm{H}_{r, h}$. For arbitrary $h \in \{\,0, \ldots, r \}$, it is easy to show that, for $r = 0, 1$, $\mathrm{H}_{r, h}$ is true. If a positive word $U_1$ is transformed into $U_2$ by using $t$ single applications of the defining relations of $G^+_{m, n}$, then the whole transformation will be said to be of \emph{chain-length} $t$. For induction hypothesis, we assume\\
$(\mathrm{A})$\,\,$\mathrm{H}_{s, h}$ is true for $0$ $\leq$  $h$ $\leq$ $s$ $\leq$ $r$ for transformations of all chain-lengths, \\
and\\
$(\mathrm{B})$\,\,$\mathrm{H}_{r+1, h}$ is true for $0$ $\leq$ $h$ $\leq$ $r+1$ for all chain-lengths $\leq$ $t$.\\
We will show the theorem $\mathrm{H}_{r+1, h}$ for chain-lengths $t+1$. For the sake of simplicity, we devide the proof into two steps.\\
{\bf Step 1}.\,$\mathrm{H}_{r+1, h}$ for $h = r+1$\\
Let $X, Y'$ be of word-length $r+1$, and let
\[
v_1 X \deeq v_2 W_2 \deeq \cdots \,\deeq v_{t+1} W_{t+1} \deeq v_{t+2} Y'
\]
be a sequence of single transformations of $t+1$ steps, where $v_1, \ldots, v_{t+2} \in L_0$ and $W_2, \ldots , W_{t+1}$ are positive words of length $r+1$. By the assumption $t > 1$, there exists an index $\tau \in \{\,2, \ldots, t+1\}$ such that we can decompose the sequence into two steps
\[
v_1 X \deeq v_{\tau} W_{\tau} \deeq v_{t+2} Y',
\]
in which each step satisfies the induction hypothesis $(\mathrm{B})$.\par
If there exists $\tau$ such that $v_{\tau}$ is equal to either to $v_1$ or $v_{t+2}$, then by induction hypothesis, $W_{\tau}$ is equivalent either to $X$ or to $Y'$. Hence, we obtain the statement for the $v_1 X \deeq v_{t+2} Y'$. Thus, we assume from now on $v_{\tau} \not= v_1, v_{t+2}$ for $1 < \tau \leq t+1$.\par
Suppose $v_1 = v_{t+2}$. If there exists $\tau$ such that $(\,v_1 = v_{t+2}, v_{\tau}\, ) \not= (\,t_i, t_j\, ), (\,u_i, u_j\, )$, then each of the equivalences says the existence of $\alpha, \beta \in L_0$ and words $Z_1, Z_2$ such that $X \deeq \alpha Z_1$, $W_{\tau} \deeq \beta Z_1 \deeq \beta Z_2$ and $Y' \deeq \alpha Z_2$. Applying the induction hypothesis $(\mathrm{A})$ to $\beta Z_1 \deeq \beta Z_2$, we get $Z_1 \deeq Z_2$. Hence, we obtain the statement $X \deeq \alpha Z_1 \deeq \alpha Z_2 \deeq Y'$. Thus, we exclude these cases from our considerations. Next, we consider the case $(\,v_1 = v_{t+2}, v_{\tau}\, ) = (\,t_i, t_j\, )$. However, because of the above consideration, we have only the case $v_2, \ldots, v_{t+1} \in L_1$. Hence, we consider the case $\tau = 1$, namely
\[
t_i X \deeq t_j W_1 \deeq t_i Y'.
\]
Applying the induction hypothesis $(\mathrm{B})$ to each step, we say that there exist words $Z_3, Z_4$ and $w(\underline{u})$ in $F^{+}_{2, \mathrm{rm}}$ such that
\[
X \deeq \Delta_{1, t_i}\cdot Z_3,\,\, W_1 \deeq \Delta_{1, t_j} \cdot Z_3,\,\,\,\,\,\,\,\,\,\,\,\,\,\,\,\,\,\,\,\,\,\,\,\,\,\,\,\,\,\,
\]
\[
W_1 \deeq w(\underline{u})\cdot \Delta_{1, t_j}\cdot Z_4,\,\, Y' \deeq w(\underline{u})\cdot \Delta_{1, t_i} \cdot Z_4.
\]
Moreover, we say that 
\[
\Delta_{1, t_j} \cdot Z_3 \deeq w(\underline{u})\cdot \Delta_{1, t_j}\cdot Z_4.\, \cdots\,(\ast) 
\]
By induction hypothesis, we have
\[
s \cdot t_1 \cdots t_{j-1} \cdot Z_3 \deeq w(\underline{u})\cdot s 
\cdot t_1 \cdots t_{j-1}\cdot Z_4.
\]
We consider the case $w(\underline{u}) \not= \varepsilon$. Applying the induction hypothesis to this equation, we say that there exists a word $Z_5$ such that 
\[
t_1 \cdots t_{j-1} \cdot Z_3 \deeq \Delta_{2, s} \cdot \mathrm{R}(w(\underline{u})) \cdot Z_5,
\]
\begin{equation}
 s \cdot t_1 \cdots t_{j-1} \cdot Z_4 \deeq \Delta_{2, \mathrm{C}(w(\underline{u})) }\cdot Z_5.\,\,\,\,
\end{equation}
Moreover, we say that there exists a word $Z_6$ such that 
\[
Z_3 \deeq \Delta_{2, s} \cdot \mathrm{R}(w(\underline{u})) \cdot Z_6,
\]
\begin{equation}
 Z_5 \deeq t_1 \cdots t_{j-1} \cdot Z_6.\,\,\,\,\,\,\,\,\,\,\,\,\,
\end{equation}
Applying (4.2) to the equation (4.1), we have
\begin{equation}
s \cdot t_1 \cdots t_{j-1} \cdot Z_4 \deeq \Delta_{2, \mathrm{C}(w(\underline{u})) }\cdot t_1 \cdots t_{j-1} \cdot Z_6.
\end{equation}
We consider the following two cases.\par
Case 1: $\mathrm{C}(w(\underline{u})) \deeq u_a \cdots u_n$ for some integer $a \ge 2$\\
From (4.3), we obtain $Z_4 \deeq u_1 \cdots u_{a-1} \cdot Z_6$. Then, we have
\[
X \deeq \Delta_{1, t_i} \cdot \Delta_{2, s} \cdot \mathrm{R}(w(\underline{u})) \cdot Z_6 \deeq \mathrm{R}(w(\underline{u})) \cdot \Delta_{1, t_i} \cdot \Delta_{2, s}\cdot Z_6,\,\,\,\,\,\,\,\,
\]
\[
 Y' \deeq w(\underline{u}) \cdot \Delta_{1, t_i} \cdot u_1 \cdots u_{a-1} \cdot Z_6 \deeq \mathrm{R}(w(\underline{u})) \cdot \Delta_{1, t_i} \cdot \Delta_{2, s}\cdot Z_6.
\]
\par
Case 2: $\mathrm{C}(w(\underline{u})) \deeq u_a \cdots u_b$ for some integers $a, b \,\,(2 \leq a \leq b < n)$\\
We consider the equation 
\[
s \cdot t_1 \cdots t_{j-1} \cdot Z_4 \deeq u_{b+1} \cdots u_n \cdot s \cdot u_1 \cdots u_{a-1} \cdot t_1 \cdots t_{j-1} \cdot Z_6.
\]
By applying the induction hypothesis to this equation, we say that there exists a word $Z_7$ such that 
\[
t_1 \cdots t_{j-1} \cdot Z_4 \deeq \Delta_{2, s} \cdot Z_7,\,\,\,\,\,\,\,\,\,\,\,\,\,\,\,\,\,\,\,\,\,\,\,\,\,\,\,\,\,\,\,\,\,\,\,\,\,\,\,\,\,\,\,\,\,\,\,\,\,\,\,\,\,\,
\]
\begin{equation}
s \cdot u_1 \cdots u_{a-1} \cdot t_1 \cdots t_{j-1} \cdot Z_6 \deeq s \cdot  u_1 \cdots u_{b}  \cdot Z_7.
\end{equation}
Moreover, we say that there exists a word $Z_8$ such that 
\begin{equation}
Z_4 \deeq \Delta_{2, s} \cdot Z_8,\,\, Z_7 \deeq  t_1 \cdots t_{j-1} \cdot Z_8.
\end{equation}
Applying (4.5) to the equation (4.4), we have
\[
Z_6 \deeq  u_a \cdots u_b \cdot Z_8.
\]
Then, we have
\[
X \deeq \Delta_{1, t_i} \cdot \Delta_{2, s} \cdot \mathrm{R}(w(\underline{u})) \cdot u_a \cdots u_b \cdot Z_8 \deeq \mathrm{R}(w(\underline{u})) \cdot \Delta_{1, t_i} \cdot \Delta_{2, s} \cdot \mathrm{C}(w(\underline{u})) \cdot Z_8,
\]
\[
Y' \deeq w(\underline{u}) \cdot \Delta_{1, t_i} \cdot \Delta_{2, s} \cdot Z_8 \deeq \mathrm{R}(w(\underline{u})) \cdot \Delta_{1, t_i} \cdot \Delta_{2, s} \cdot \mathrm{C}(w(\underline{u})) \cdot Z_8.\,\,\,\,\,\,\,\,\,\,\,\,\,\,\,\,\,\,\,\,\,\,\,\,\,\,\,\,\,\,\,\,\,
\]
In the case of $(\,v_1 = v_{t+2}, v_{\tau}\, ) = (\,u_i, u_j\, )$, we can prove the statement in a similar manner.
\par
Suppose $v_1 \not= v_{t+2}$. We consider the following three cases.
\par
Case 1: $(\,v_1, v_{t+2}\, ) = (\,t_i, t_k\, ), (\,u_i, u_k\, )$\\
We consider the case $(\,v_1, v_{t+2}\, ) = (\,t_i, t_k\, )$. Then, we can easily show the case $v_{\tau} = s, u_j$. Thus, we have only the case $v_2, \ldots, v_{t+1} \in L_1$. Hence, we consider the case $\tau = 1$, namely
\[
t_i X \deeq t_j W_1 \deeq t_k Y'.
\]
Applying the induction hypothesis to each step, we say that there exist words $Z_1, Z_{2}$ and $w(\underline{u})$ in $F^{+}_{2, \mathrm{rm}}$ such that
\[
X \deeq \Delta_{1, t_i}\cdot Z_1,\,\, W_1 \deeq \Delta_{1, t_j} \cdot Z_1,\,\,\,\,\,\,\,\,\,\,\,\,\,\,\,\,\,\,\,\,\,\,\,\,\,\,\,\,\,\,\,\,\,
\]
\[
W_1 \deeq w(\underline{u})\cdot \Delta_{1, t_j}\cdot Z_{2},\,\, Y' \deeq w(\underline{u})\cdot \Delta_{1, t_k} \cdot Z_{2}.
\]
Thus, we say that $\Delta_{1, t_j} \cdot Z_1 \deeq w(\underline{u})\cdot \Delta_{1, t_j}\cdot Z_{2}$. Since this equation has the same form as the equation $(\ast)$, we can find the solution in a similar way. Hence, we verify the statement in the case  $(\,v_1, v_{t+2}\, ) = (\,t_i, t_k\, )$. In the same way, we verify the statement in the case $(\,v_1, v_{t+2}\, ) = (\,u_i, u_k\,)$.
\par
Case 2: $(\,v_1, v_{t+2}\, ) = (\,s, t_j\, ), (\,s, u_j\,)$\\
We consider the case $(\,v_1, v_{t+2}\, ) = (\,s, t_j\,)$. If $v_{t+1} = t_i$, then, by applying the induction hypothesis, we easily show the statement. Thus, we consider the case $(\,v_1, v_{t+1}, v_{t+2}\,) = (\,s, u_i, t_j \,)$, namely
\[
s X \deeq u_i W_{t+1} \deeq t_j Y'.
\]
Applying the induction hypothesis to each step, we say that there exist words $Z_{1}$ and $Z_{2}$ such that
\[
X \deeq \Delta_{2, s}\cdot Z_{1},\,\, W_{t+1} \deeq \Delta_{2, u_i} \cdot Z_{1},
\]
\[
W_{t+1} \deeq t_j \cdot Z_{2},\,\, Y' \deeq u_i \cdot Z_{2}.\,\,\,\,\,\,\,\,\,\,\,\,\,\,\,\,
\]
Thus, we say that $\Delta_{2, u_i} \cdot Z_{1} \deeq t_j \cdot Z_{2}$. By applying the induction hypothesis, there exists a word $Z_{3}$ such that
\begin{equation}
Z_{2} \deeq u_{i+1} \cdots u_n \cdot Z_{3},\,\, s \cdot u_1 \cdots u_{i-1} \cdot Z_{1} \deeq t_j \cdot Z_{3}.
\end{equation}
By applying the induction hypothesis to the equation $(4.6)$, we say that there exists a word $Z_{4}$ such that 
\[
 u_1 \cdots u_{i-1} \cdot Z_{1} \deeq  \Delta_{1, s} \cdot Z_{4},\,\,Z_{3} \deeq \Delta_{1, t_j} \cdot Z_{4}.
\]
Moreover, we say that there exists a word $Z_{5}$ such that 
\[
Z_{1} \deeq \Delta_{1, s} \cdot Z_{5},\,\,Z_{4} \deeq u_1 \cdots u_{i-1} \cdot Z_{5}.
\]
Thus, we have
\[
X \deeq \Delta_{2, s} \cdot \Delta_{1, s} \cdot Z_{5} \deeq \Delta_{1, s} \cdot \Delta_{2, s} \cdot Z_{5},\,\,\,\,\,\,\,\,\,\,\,\,\,\,\,\,\,\,\,\,\,\,\,\,\,\,\,\,\,\,\,\,\,\,\,\,\,\,\,\,\,\,\,\,\,\,
\]
\[
Y' \deeq u_i \cdots u_{n} \cdot \Delta_{1, t_j} \cdot u_1 \cdots u_{i-1} \cdot Z_{5} \deeq \Delta_{1, t_j} \cdot \Delta_{2, s} \cdot Z_{5}.
\]
We verify the statement in the case $(\,v_1, v_{t+2}\,) = (\,s, u_j\,)$ in a similar manner.
\par
Case 3: $(\,v_1, v_{t+2}\, ) = (\,t_i, u_j\, )$\\
First, we assume that there exists an index $\tau$ such that $v_{\tau}$ is equal to $s$. Then, we consider the case $(\,v_1, v_{\tau}, v_{t+2}\, ) = (\,t_i, s, u_j\, )$, namely
\[
t_i X \deeq s W_{\tau} \deeq  u_j Y'.
\]
Applying the induction hypothesis to each step, we say that there exist words $Z_1$ and $Z_2$ such that
\[
X \deeq \Delta_{1, t_i} \cdot Z_1,\,\, W_{\tau} \deeq \Delta_{1, s} \cdot Z_1,
\]
\[
W_{\tau} \deeq \Delta_{2, s} \cdot Z_2,\,\, Y' \deeq \Delta_{2, u_j} \cdot Z_2.
\]
Moreover, we say that 
\[
\Delta_{1, s} \cdot Z_1 \deeq \Delta_{2, s} \cdot Z_2.
\]
Applying the induction hypothesis to this equation, we say that there exists a word $Z_3$ such that
\[
Z_1 \deeq \Delta_{2, s} \cdot Z_3,\,\,Z_2 \deeq \Delta_{1, s} \cdot Z_3.
\]
Thus, we have
\[
X \deeq \Delta_{1, t_i} \cdot \Delta_{2, s} \cdot Z_3 \deeq u_j \cdot u_{j+1} \cdots u_n \cdot \Delta_{1, t_i} \cdot u_1 \cdots u_{j-1} \cdot Z_3,
\]
\[
Y' \deeq \Delta_{2, u_j} \cdot \Delta_{1, s} \cdot Z_3 \deeq t_i \cdot u_{j+1} \cdots u_n \cdot \Delta_{1, t_i} \cdot u_1 \cdots u_{j-1} \cdot Z_3
.
\]
Thus, in the consideration of Case 3, we assume from now on $v_{\tau} \not= s$ for $1 < \tau \leq t+1$.
We consider the following three cases.
\par
Case 3 -- 1: $(\,v_1, v_2, v_{t+2}\,) = (\,t_i, t_k, u_j\,)$\\
We consider the case
\[
t_i X \deeq t_k W_{2} \deeq  u_j Y'.
\]
Applying the induction hypothesis to each step, we say that there exist words $Z_1$ and $Z_2$ such that
\[
X \deeq \Delta_{1, t_i} \cdot Z_1,\,\, W_{2} \deeq \Delta_{1, t_k} \cdot Z_1,
\]
\[
W_{2} \deeq u_j \cdot Z_2,\,\, Y' \deeq t_k \cdot Z_2.\,\,\,\,\,\,\,\,\,\,\,\,\,\,\,\,\,
\]
Moreover, we obtain an equation $\Delta_{1, t_k} \cdot Z_1 \deeq u_j \cdot Z_2$. Then, there exists a word $Z_3$ such that 
\[
Z_2 \deeq t_{k+1} \cdots t_{m} \cdot Z_3,\,\, s \cdot t_1 \cdots t_{k-1} \cdot Z_1 \deeq u_j \cdot Z_3.
\]
By the induction hypothesis, we say that there exists a word $Z_4$
\[
t_1 \cdots t_{k-1} \cdot Z_1 \deeq \Delta_{2, s} \cdot Z_4,\,\,Z_3 \deeq \Delta_{2, u_j} \cdot Z_4.
\]
Moreover, we say that there exists a word $Z_5$ such that
\[
Z_1 \deeq \Delta_{2, s} \cdot Z_5,\,\,Z_4 \deeq t_1 \cdots t_{k-1} \cdot Z_5.
\]
Thus, we have
\[
X \deeq \Delta_{1, t_i} \cdot \Delta_{2, s} \cdot Z_5 \deeq u_j \cdot u_{j+1} \cdots u_n \cdot \Delta_{1, t_i} \cdot u_1 \cdots u_{j-1} \cdot Z_5,\,\,\,\,\,\,\,\,\,\,\,\,\,\,\,\,\,\,\,\,\,\,\,\,\,\,\,\,\,\,\,\,\,\,\,\,\,\,\,\,\,\,\,\,\,\,\,\,\,\,\,\,\,\,\,\,
\]
\[
Y' \deeq t_k \cdot t_{k+1} \cdots t_m \cdot \Delta_{2, u_j} \cdot t_1 \cdots t_{k-1} \cdot Z_5 \deeq t_i \cdot u_{j+1} \cdots u_n \cdot \Delta_{1, t_i} \cdot u_1 \cdots u_{j-1} \cdot Z_5.
\]
\par
Case 3 -- 2: $(\,v_1, v_{t+1}, v_{t+2}\,) = (\,t_i, u_k, u_j\,)$\\
In the same way as the Case 3 -- 1, we verify the statement in this case.
\par
Case 3 -- 3: $(\,v_1, v_2, v_{t+1}, v_{t+2}\,) = (\,t_i, u_{j_1}, t_{i_1}, u_j\,)$\\
We consider the case
\[
t_i X \deeq t_{i_1} W_{t+1} \deeq  u_j Y'.
\]
Applying the induction hypothesis to each step, we say that there exist words $Z_1, Z_2$ and $w(\underline{u})$ in $F^{+}_{2, \mathrm{rm}}$ such that
\[
X \deeq w(\underline{u}) \cdot \Delta_{1, t_i} \cdot Z_1,\,\, W_{t+1} \deeq w(\underline{u}) \cdot \Delta_{1, t_{i_1}} \cdot Z_1,
\]
\[
W_{t+1} \deeq u_j \cdot Z_2,\,\, Y' \deeq t_{i_1} \cdot Z_2.\,\,\,\,\,\,\,\,\,\,\,\,\,\,\,\,\,\,\,\,\,\,\,\,\,\,\,\,\,\,\,\,\,\,\,\,\,\,\,\,\,\,\,\,\,\,\,\,\,\,
\]
Moreover, we say that $w(\underline{u}) \cdot \Delta_{1, t_{i_1}} \cdot Z_1 \deeq u_j \cdot Z_2$. And we say that there exists a word $Z_3$ such that
\[
Z_2 \deeq t_{i_{1}+1} \cdots t_m \cdot Z_3,\,\,w(\underline{u}) \cdot s \cdot t_{1} \cdots t_{i_{1}-1} \cdot Z_1 \deeq u_j \cdot Z_3.
\]
Here, we consider the case $u_j \not|_l \, w(\underline{u})$. By applying the induction hypothesis, we say that there exist a word $Z_4$ and $w(\underline{t})$ in $F^{+}_{1, \mathrm{rm}}$ such that
\[
s \cdot t_{1} \cdots t_{i_{1}-1} \cdot Z_1 \deeq w(\underline{t}) \cdot \Delta_{2, \mathrm{C}(w(\underline{u})) } \cdot Z_4,\,\,Z_3 \deeq w(\underline{t}) \cdot \Delta_{2, u_j} \cdot \mathrm{R}(w(\underline{u})) \cdot Z_4.
\]
By applying the induction hypothesis to this equation, we say that there exists a word $Z_5$ such that
\begin{equation}
t_{1} \cdots t_{i_{1}-1} \cdot Z_1 \deeq \Delta_{1, s} \cdot \mathrm{R}(w(\underline{t})) \cdot Z_5,\,\,\Delta_{2, \mathrm{C}(w(\underline{u})) } \cdot Z_4 \deeq \Delta_{1, \mathrm{C}(w(\underline{t})) } \cdot Z_5.
\end{equation}
We can find a general solution of the equation (4.7)
\[
Z_4 \deeq  \Delta_{1, s} \cdot \mathrm{C}(w(\underline{u})) \cdot Z_6,\,\,Z_5 \deeq  \Delta_{2, s} \cdot \mathrm{C}(w(\underline{t})) \cdot Z_6.
\]
Thus, we have
\[
X \deeq w(\underline{u}) \cdot \Delta_{1, t_i} \cdot \Delta_{2, s} \cdot t_{i_1} \cdots t_{m} \cdot w(\underline{t}) \cdot Z_6\,\,\,\,\,\,\,\,\,\,\,\,\,\,\,\,\,\,\,\,\,\,\,\,\,\,\,\,\,\,\,\,\,\,\,\,\,\,\,\,\,\,\,\,\,\,\,\,\,\,\,\,\,\,\,
\]
\[
\deeq u_j \cdot u_{j+1} \cdots u_n \cdot \Delta_{1, t_i} \cdot u_1 \cdots u_{j-1} \cdot w(\underline{u}) \cdot t_{i_1} \cdots t_m \cdot w(\underline{t}) \cdot Z_6,\,\,\,
\]
\[
Y' \deeq t_{i_1} \cdot t_{i_{1}+1} \cdots t_m \cdot w(\underline{t}) \cdot \Delta_{2, u_j} \cdot \mathrm{R}(w(\underline{u})) \cdot \Delta_{1, s} \cdot \mathrm{C}(w(\underline{u})) \cdot Z_6
\]
\[
\deeq t_i \cdot u_{j+1} \cdots u_n \cdot \Delta_{1, t_i} \cdot u_1 \cdots u_{j-1} \cdot w(\underline{u}) \cdot t_{i_1} \cdots t_m \cdot w(\underline{t}) \cdot Z_6.\,\,\,\,\,\,
\]
{\bf Step 2}.\,$\mathrm{H}_{r+1, h}$ for $0 \leq h \leq r+1$\\
We put $h' := r+1-h$. We will show the general theorem $\mathrm{H}_{r+1, h}$ by induction on $h'$. The case $h' = 0$ is proved in Step 1. We assume $h'=0, \ldots, r-h$. Let $X$ be of word-length $r+1$, and let $Y^{(h)}$ be of word-length $h$. We consider a sequence of single transformations of $t+1$ steps
\begin{equation}
v_1 X \deeq \cdots \,\deeq V \cdot Y^{(h)},
\end{equation}
where $v_1 \in L_0$ and $V$ is a positive word of length $r-h+2$. To show the theorem $\mathrm{H}_{r+1, h}$ (i.e. $h'=r-h+1$ ) for chain-lengths $t+1$, we discuss the cases $(v_1, V) = (s, w(\underline{t})\,), (s, w(\underline{u})\,), (t_i, w(\underline{t})\,), (u_i, w(\underline{u})\,)$. 
\par
Case 1:\,$(v_1, V) = (s, w(\underline{t})\,), (s, w(\underline{u})\,)$.\\
First, we discuss the case $(v_1, V) = (s, w(\underline{t})\,)$ and decompose $w(\underline{t})$ into $w_{1}(\underline{t}) \cdot t_{a}$ (i.e. $w(\underline{t}) \equiv  w_{1}(\underline{t}) \cdot t_{a}$ ). We consider the case
\[
s X \deeq \cdots \,\deeq w_{1}(\underline{t})\cdot t_a \cdot Y^{(h)}.
\]
Applying the induction hypothesis, we say that there exists a word $Z_1$ such that
\[
X \deeq \Delta_{1, s} \cdot \mathrm{R}(w_1(\underline{t}))\cdot Z_1,\,\, t_{a}\cdot Y^{(h)}  \deeq \Delta_{1, \mathrm{C}(w_1(\underline{t}))} \cdot Z_1.\, \cdots\,(\ast \ast) 
\]
We consider the following two cases.
\par
Case 1 -- 1:\,$\mathrm{C}(w_1(\underline{t})) \deeq t_b \cdots t_m$ for some integer $b \ge 2$\\
The equation is the following
\[
t_{a}\cdot Y^{(h)}  \deeq s \cdot t_1 \cdots t_{b-1} \cdot Z_1. 
\]
Applying the induction hypothesis, we say that there exists a word $Z_2$ such that
\[
Y^{(h)} \deeq \Delta_{1, t_{a}} \cdot Z_2,\,\, Z_1 \deeq t_b \cdots t_{m} \cdot Z_2. 
\]
Thus, we have
\[
X \deeq \Delta_{1, s} \cdot \mathrm{R}(w_1(\underline{t})) \cdot t_b \cdots t_{m} \cdot Z_2 \deeq  \Delta_{1, s} \cdot \mathrm{R}(w_1(\underline{t}) \cdot t_a) \cdot Z_2, 
\]
\[
Y^{(h)} \deeq \Delta_{1, t_{a}} \cdot Z_2 \deeq \Delta_{1, \mathrm{C}(w_1(\underline{t}) \cdot t_a)} \cdot Z_2.\,\,\,\,\,\,\,\,\,\,\,\,\,\,\,\,\,\,\,\,\,\,\,\,\,\,\,\,\,\,\,\,\,\,\,\,\,\,\,\,\,\,\,\,\,\,\,\,\,\,\,\,\,\,\,\,\, 
\]
\par
Case 1 -- 2:\,$\mathrm{C}(w_1(\underline{t})) \deeq t_b \cdots t_c$ for some integers $b, c \,\,(2 \leq b \leq c < m)$\\
The equation is the following
\[
t_{a}\cdot Y^{(h)}  \deeq t_{c+1} \cdots t_m \cdot s \cdot t_1 \cdots t_{b-1} \cdot Z_1. 
\]
We discuss the case $t_{a} \not= t_{c+1}$. Applying the induction hypothesis, we say that there exist a word $Z_2$ and $w(\underline{u})$ in $F^{+}_{2, \mathrm{rm}}$ such that
\[
Y^{(h)}  \deeq w(\underline{u}) \cdot \Delta_{1, t_{a}} \cdot Z_2,\,\, s \cdot t_1 \cdots t_{b-1} \cdot Z_1 \deeq w(\underline{u}) \cdot s \cdot t_1 \cdots t_{c} \cdot Z_2.
\]
Moreover, we say that there exists a word $Z_3$ such that
\[
t_1 \cdots t_{b-1} \cdot Z_1 \deeq \Delta_{2, s} \cdot \mathrm{R}(w(\underline{u})) \cdot Z_3,\,\, s \cdot t_1 \cdots t_{c} \cdot Z_2 \deeq \Delta_{2, \mathrm{C}(w(\underline{u}))} \cdot Z_3.
\]
We say that there exists a word $Z_4$ such that
\[
Z_1 \deeq \Delta_{2, s} \cdot \mathrm{R}(w(\underline{u})) \cdot Z_4,\,\, Z_3 \deeq t_1 \cdots t_{b-1} \cdot Z_4.
\]
Then, we have $s \cdot t_1 \cdots t_{c} \cdot Z_2 \deeq \Delta_{2, \mathrm{C}(w(\underline{u}))} \cdot t_1 \cdots t_{b-1} \cdot Z_4$.
\par
Case 1 -- 2 -- 1:\,$\mathrm{C}(w(\underline{u})) \deeq u_d \cdots u_n$ for some integer $d \ge 2$\\
The equation is the following
\[
s \cdot t_1 \cdots t_{c} \cdot Z_2 \deeq s \cdot u_1 \cdots u_{d-1} \cdot t_1 \cdots t_{b-1} \cdot Z_4.
\]
Moreover, we say
\[
t_b \cdots t_{c} \cdot Z_2 \deeq  u_1 \cdots u_{d-1} \cdot Z_4.
\]
By the induction hypothesis, we say that there exists a word $Z_5$ such that
\[
Z_2 \deeq u_1 \cdots u_{d-1} \cdot Z_5,\,\, Z_4 \deeq t_b \cdots t_{c}\cdot Z_5.
\]
Thus, we have
\[
X \deeq \Delta_{1, s} \cdot \mathrm{R}(w_1(\underline{t}))\cdot \Delta_{2, s} \cdot \mathrm{R}(w(\underline{u})) \cdot t_b \cdots t_{c}\cdot Z_5 \deeq \Delta_{1, s} \cdot w_1(\underline{t}) \cdot \Delta_{2, s} \cdot \mathrm{R}(w(\underline{u})) \cdot Z_5
\]
\[
\deeq \Delta_{1, s} \cdot \mathrm{R}(w_1(\underline{t})\cdot t_a) \cdot \Delta_{2, s} \cdot \mathrm{R}(w(\underline{u})) \cdot Z_5,\,\,\,\,\,\,\,\,\,\,\,\,\,\,\,\,\,\,\,\,\,\,\,\,\,\,\,\,\,\,\,\,\,\,\,\,\,\,\,\,\,\,\,\,\,\,\,\,\,\,\,\,\,\,\,\,\,\,\,\,\,\,\,\,\,\,\,\,\,\,\,\,\,\,\,\,\,\,\,\,\,\,\,\,\,\,\,\,\,\,\,\,\,\,\,\,\,\,\,\,\,\,\,\,\,\,\,\,\,\,\,
\]
\[
Y^{(h)} \deeq w(\underline{u}) \cdot \Delta_{1, t_{a}} \cdot u_1 \cdots u_{d-1} \cdot Z_5 \deeq \Delta_{1, t_{a}} \cdot \Delta_{2, s} \cdot \mathrm{R}(w(\underline{u})) \cdot Z_5,\,\,\,\,\,\,\,\,\,\,\,\,\,\,\,\,\,\,\,\,\,\,\,\,\,\,\,\,\,\,\,\,\,\,\,\,\,\,\,\,\,\,\,\,\,\,\,\,\,
\]
\[
\deeq \Delta_{1, \mathrm{C}(w_1(\underline{t})\cdot t_a)} \cdot \Delta_{2, s} \cdot \mathrm{R}(w(\underline{u})) \cdot Z_5.\,\,\,\,\,\,\,\,\,\,\,\,\,\,\,\,\,\,\,\,\,\,\,\,\,\,\,\,\,\,\,\,\,\,\,\,\,\,\,\,\,\,\,\,\,\,\,\,\,\,\,\,\,\,\,\,\,\,\,\,\,\,\,\,\,\,\,\,\,\,\,\,\,\,\,\,\,\,\,\,\,\,\,\,\,\,\,\,\,\,\,\,\,\,\,\,\,\,\,\,\,\,\,\,\,\,\,\,\,\,\,\,\,\,\,\,\,\,\,\,\,\,\,\,\,
\]
\par
Case 1 -- 2 -- 2:\,$\mathrm{C}(w(\underline{u})) \deeq u_e \cdots u_f$ for some integers $e, f \,\,(2 \leq e \leq f < n)$\\
The equation is the following
\[
s \cdot t_1 \cdots t_{c} \cdot Z_2 \deeq u_{f+1}\cdots u_{n} \cdot s \cdot u_1 \cdots u_{e-1} \cdot t_1 \cdots t_{b-1} \cdot Z_4.
\]
By the induction hypothesis, we say that there exists a word $Z_5$ such that
\[
t_1 \cdots t_{c} \cdot Z_2 \deeq \Delta_{2, s} \cdot Z_5,\,\,s \cdot u_1 \cdots u_{e-1} \cdot t_1 \cdots t_{b-1} \cdot Z_4 \deeq s \cdot u_1 \cdots u_f \cdot Z_5.
\]
Moreover, we say that $t_1 \cdots t_{b-1} \cdot Z_4 \deeq u_e \cdots u_f \cdot Z_5$. By the induction hypothesis, there exists a word $Z_6$ such that
\[
Z_4 \deeq u_e \cdots u_f \cdot Z_6,\,\,Z_5 \deeq t_1 \cdots t_{b-1} \cdot Z_6.
\]
Hence, we say $t_1 \cdots t_{c} \cdot Z_2 \deeq \Delta_{2, s} \cdot t_1 \cdots t_{b-1} \cdot Z_6$. By the induction hypothesis, we show
\[
t_b \cdots t_{c} \cdot Z_2 \deeq \Delta_{2, s} \cdot Z_6.
\]
We say that there exists a word $Z_7$ such that
\[
Z_2 \deeq \Delta_{2, s} \cdot Z_7,\,\, Z_6 \deeq t_b \cdots t_{c} \cdot Z_7.
\]
Thus, we have
\[
X \deeq \Delta_{1, s} \cdot \mathrm{R}(w_1(\underline{t})) \cdot \Delta_{2, s} \cdot \mathrm{R}(w(\underline{u}))\cdot u_e \cdots u_f \cdot t_b \cdots t_{c} \cdot Z_7\,\,\,\,\,\,\,\,\,\,\,\,\,\,\,\,\,\,\,\,\,\,\,\,\,\,\,\,\,\,\,\,\,\,\,\,\,\,\,\,\,\,\,\,\,\,\,\,\,\,\,\,\,\,\,\,\,\,\,\,\, 
\]
\[
\deeq \Delta_{1, s} \cdot w_1(\underline{t}) \cdot \Delta_{2, s} \cdot \mathrm{R}(w(\underline{u}))\cdot u_e \cdots u_f \cdot Z_7 \deeq \Delta_{1, s} \cdot \mathrm{R}(w_1(\underline{t})\cdot t_a) \cdot \Delta_{2, s} \cdot w(\underline{u}) \cdot Z_7, 
\]
\[
Y^{(h)} \deeq  w(\underline{u}) \cdot \Delta_{1, t_a} \cdot \Delta_{2, s} \cdot Z_7 \deeq \Delta_{1, \mathrm{C}(w_1(\underline{t})\cdot t_a)} \cdot \Delta_{2, s} \cdot w(\underline{u}) \cdot Z_7.\,\,\,\,\,\,\,\,\,\,\,\,\,\,\,\,\,\,\,\,\,\,\,\,\,\,\,\,\,\,\,\,\,\,\,\,\,\,\,\,\,\,\,\,\,\,\,\,\,\,
\]
We can verify the statement in the case $(v_1, V) = (s, w(\underline{u}))$ in a similar manner.
\par
Case 2:\,$(v_1, V) = (t_i, w(\underline{t})\,), (u_i, w(\underline{u})\,)$.\\
First, we discuss the case $(v_1, V) = (t_i, w(\underline{t})\,)$. And we decompose $w(\underline{t})$ into $w_{1}(\underline{t}) \cdot t_{a}$ (i.e. $w(\underline{t}) \equiv  w_{1}(\underline{t}) \cdot t_{a}$ ). We consider the case
\[
t_i X \deeq \cdots \,\deeq w_{1}(\underline{t})\cdot t_a \cdot Y^{(h)}.
\]
Applying the induction hypothesis, we say that there exist a word $Z_1$ and $w(\underline{u})$ in $F^{+}_{2, \mathrm{rm}}$ such that
\[
X \deeq w(\underline{u}) \cdot \Delta_{1, t_i} \cdot \mathrm{R}(w_1(\underline{t})) \cdot Z_1,\,\, t_a Y^{(h)} \deeq w(\underline{u}) \cdot \Delta_{1, \mathrm{C}(w_1(\underline{t}))} \cdot Z_1.
\]
By the induction hypothesis, we say that $w(\underline{u}) \,|_l \, Y^{(h)}$. Hence, we write $Y^{(h)} \deeq w(\underline{u})\cdot \widetilde{Y}^{(h)}$. Then, we consider an equation
\[
t_a \widetilde{Y}^{(h)} \deeq \Delta_{1, \mathrm{C}(w_1(\underline{t}))} \cdot Z_1.
\]
Since this equation has the same form as the equation $(\ast \ast)$, we can find the solution in a similar way. Hence, we verify the statement in the case  $(\,v_1, V\, ) = (\,t_i, \, w(\underline{t})\,)$. In the same way, we verify the statement in the case $(\,v_1, V\, ) = (\,u_i, \,w(\underline{u}) \,)$.
 \end{proof}
This completes the proof of Theorem 4.1.
 \end{proof}
\ \ \\
\begin{remark}{\it The presentation of the monoid $G^{+}_{m, n}$ is not \emph{complete} (\cite{[D2]}). Furthermore, we easily show that the process of \emph{completion} (\cite{[D2]}) does not finish in finite steps.}
\end{remark}
\section{Some decision problems on the group $G_{m, n}$}
 In this section, we will solve the word problem in the group $G_{m, n}$ and determine the center of it, by showing the monoid $G^{+}_{m, n}$ injects in the group $G_{m, n}$. \par %\\ Editor
 Since the existence of a fundamental element and the cancellativity of the monoid $G^{+}_{m, n}$ have be shown, we show the following by applying the Lemma 2.3 to this case.
\ \ \\
\begin{proposition}
%\label{mythm}
{\it The localization homomorphism $\pi:  G^{+}_{m, n}\to G_{m, n}$ is injective.}\par 
\end{proposition}
\par
\begin{proposition}
%\label{mythm}
{\it The word problem in the group $G_{m, n}$ can be solved. }
 \par 
\end{proposition}
Thanks to Theorem 4.1, we can also show the following proposition.
\ \ \\
\begin{proposition}
%\label{mythm}
%\it Editor
{\it The monoid $G_{m, n}^+$ does not always have least common multiples. }
\end{proposition}
\begin{proof} Due to the Theorem 4.1, we say, for example,
\[
\mathrm{mcm}_{r}(\{ t_1, t_2 \}) = \{ w(\underline{u}) \cdot \Delta_1 \mid w(\underline{u}) \in F^{+}_{2, \mathrm{rm}} \}.
\]

\end{proof}
As a consequence of Proposition 5.3, the monoid $G_{m, n}^{+}$ is neither Garside nor Artin monoid. We have an important remark on the monoid $G_{m, n}^{+}$.
\ \ \\
\begin{remark}{\it  For each letter $v$ in $L_0$, both sides of the defining relations of $G^{+}_{m, n}$ contain the same number of the letter $v$. For arbitrary word $W$ in $G^{+}_{m, n}$, the number of the letter $v$ in $W$ ought to be preserved in the process of rewriting $W$.}
\end{remark}
\ \ \\
\begin{proposition}
%\label{mythm}
%\it Editor
{\it The center $\mathcal{Z}(G_{m, n})$  is isomorphic to $\Z$ and generated by $\Delta$.} 
 \par 
\end{proposition}
\begin{proof} First, we prove the following two Claims.\par
{\bf Claim 1.}  
$\delta \in \mathcal{Z}(G^+_{m, n})\setminus \{\varepsilon\}$ $\Rightarrow$ $\Delta |_l \, \delta$.

\begin{proof} Thanks to the Theorem 4.1, it is easy to show that $\delta$ contains at least one letter except for the letter $s$. Hence, there exist a non-negative integer $k$ in $\Z_{\ge0}$ and a letter $v$ in $L_1 \cup L_2$ such that 
\[ 
\delta \deeq  s^{k} \cdot v \cdot d
\] 
for some positive word $d$. Since $\delta$ belongs to the center, an equation $s \cdot \delta \deeq \delta \cdot s$ holds. By using the cancellativity of the monoid $G^+_{m, n}$, we have
\[ 
s \cdot v \cdot d \deeq v \cdot d \cdot s. 
\] 
By the Theorem 4.1, we easily show that
\[ 
\Delta_{1, s} \,|_l \, \delta\,\,\, \mathrm{or}\,\,\,  \Delta_{2, s} \,|_l \,\delta. 
\] 
Without loss of generality, we assume that $\Delta_{1, s} \,|_l \, \delta$. Hence, there exists a positive word $\delta_1$ such that
\[
\delta \deeq \Delta_{1, s} \cdot \delta_1.
\]
We easily show that $\delta_1 \not= \varepsilon$ and $\delta_1$ contains at least one letter except for the letter $s$. Due to the cancellativity, we have $s \cdot \delta_1 \deeq \delta_1 \cdot s$. In the same way, we can show
\[
\Delta_{1, s} \,|_l \, \delta_1\,\,\, \mathrm{or}\,\,\,  \Delta_{2, s} \,|_l \,\delta_1.
\]
From the Theorem 4.1, we say that $\delta$ cannot be a power of $\Delta_{1, s}$. Hence, there exists a positive integer $j$ such that $\Delta^{j}_{1, s} \cdot \Delta_{2, s} \,|_l \, \delta$. Then, there exists a positive word $\delta_2$ such that\[
\delta \deeq \Delta^{j}_{1, s} \cdot \Delta_{2, s} \cdot \delta_2.
\]
We devide $\delta_2$ by $\Delta_{1, s}$ and $\Delta_{2, s}$ as much as we can. Namely, there exist positive integers $j_1, j_2$ and a positive word $\delta_3$ such that 
\[
\delta \deeq \Delta^{j_{1}}_{1, s} \cdot \Delta^{j_{2}}_{2, s} \cdot \delta_3 \,\,\mathrm{and}\,\, \Delta_{1, s}, \Delta_{2, s}\,\not|_l \, \delta_3. 
\]
We easily show that $\delta_3 \not= \varepsilon$. If $s \,\not|_l \, \delta_3$, then , from the above consideration, we show that
\[
\Delta_{1, s} \,|_l \, \delta_3\,\,\, \mathrm{or}\,\,\,  \Delta_{2, s} \,|_l \,\delta_3. 
\]
A contradiction. Hence, we say that $s \,|_l \, \delta_3$. Thus, we have
\[
\Delta |_l \, \delta.
\]
\end{proof}
{\bf Claim 2.}  
The center $\mathcal{Z}(G^{+}_{m, n})$ is isomorphic to an infinite cyclic monoid and generated by $\Delta$.
\begin{proof} We take an element $\delta$ in $\mathcal{Z}(G^+_{m, n})\setminus \{\varepsilon\}$. By applying the Claim 1 repeatedly, we say that there exists a positive integer $j$ such that
\[
\delta \deeq \Delta^{j}.
\]
 \end{proof}
Next, for an arbitrary element $V$ in $\mathcal{Z}(G_{m, n})$, there exists a non-negative integer $k$ in $\Z_{\ge0}$ such that $\Delta^{k} \cdot V$ is equivalent to a positive word. Since the localization homomorphism $\pi$ is injective, there exists a unique element $V'$ in $G^{+}_{m, n}$ such that $\pi(V') = \Delta^{k} \cdot V$. The element $V'$ belongs to the center $\mathcal{Z}(G^{+}_{m, n})$. Due to the Claim 2, we show that there exists a positive integer $k'$ such that
\[
\Delta^{k'} = \Delta^{k} \cdot V.
\]
Hence, we have $V = \Delta^{k'-k}$.
 \end{proof}
%\vspace{-4pt} Editor
%これは、今のところ利用していない。\\
%\noindent
%{\bf Fact 2.}
%\emph{A non-trivial $2\times 2$-matrix representation 

%\centerline
%{$\rho\ :  \ G_{\mathrm{B_{ii}}} \longrightarrow \mathrm{GL}(2,\C)$}
%\noindent
%is induced by the correspondence  $a\mapsto A, b\mapsto B, c\mapsto C$, 
%where matrices $A, B, C$ are the following.\\
%\,\,\,\,\,\,\,\,\,\,\,\,\,\,\,\,\,\,\,\,\,\,\,\,\,\,\,
%{\small     
%$
% A=u\left( \begin{array}{cc}
%1&l^2\\0&1
%\end{array}
%\right)
%, \quad 
%B=v\left( \begin{array}{cc}
%l&0\\0&l^{-1}
%\end{array}
%\right), \quad 
%C=u\left( \begin{array}{cc}
%1&1\\0&1
%\end{array}
%\right),
%$
%}

%

%\bigskip
%

%

%
%\noindent%Editor

%
%\noindent%Editor
\emph{Acknowledgement.}\! 
%\begin{acknowledgement}
The author is deeply grateful to Kyoji Saito and Ichiro Shimada for very interesting discussions and their warm encouragement. The author thanks Masahiko Yoshinaga and Yusuke Kuno for helpful discussions and valuable comments. This research is supported by JSPS Fellowships for Young Scientists $(23\cdot10023)$. This researsh is also supported by World Premier International Research Center Initiative (WPI Initiative), MEXT, Japan.
%\end{acknowledgement}

\begin{flushright}
\begin{small}
%Kavli IPMU, \\
%University of Tokyo, \\
%Kashiwa, Chiba 277-8583 Japan \\

Department of Mathematical Sciences, \\
University of Tokyo, \\
3-8-1 Komaba Meguro-ku Tokyo, 153-8914 Japan \\
e-mail address :  tishibe@ms.u-tokyo.ac.jp
\end{small}
\end{flushright}
\end{document}